\documentclass[10pt,twoside,a4paper,english,reqno]{amsart}
\usepackage[dvips]{epsfig}
\usepackage{lmodern}
\usepackage{amscd}
\usepackage{physics}
\usepackage{caption}
\usepackage{a4wide}
\usepackage{amssymb,amsthm,amsmath,amsfonts,mathrsfs}
\usepackage{yhmath}
\usepackage{bm,latexsym,enumitem,dsfont,tabularx,graphicx}
\usepackage[T1]{fontenc}
\usepackage[utf8]{inputenc}
\usepackage{url}
\usepackage{hyperref}
\usepackage{doi}
\usepackage[only,shortleftarrow,shortrightarrow]{stmaryrd}
\usepackage{extpfeil}
\hypersetup{
    colorlinks=true,
    linkcolor=blue,
    filecolor=red,
    urlcolor=blue,
    citecolor=red,
}
\usepackage{nicefrac,microtype,upref,ulem,footmisc}

\theoremstyle{plain}
\newtheorem{theorem}{Theorem}[section]

\newtheorem{lemma}[theorem]{Lemma}

\newtheorem{remark}[theorem]{Remark}

\numberwithin{equation}{section}

\newcommand{\R}{\mathbb R}

\numberwithin{figure}{section}
\numberwithin{table}{section}

\title[A FSG method for fZK equation]{A numerical method for the fractional Zakharov-Kuznetsov equation}

\author{Mukul Dwivedi and Andreas Rupp}
\address{Department of Mathematics, Saarland University, Saarbr\"ucken, Germany}
\email{\{mukul.dwivedi;andreas.rupp\}@uni-saarland.de}


\begin{document}

\begin{abstract}
This paper develops a fully discrete Fourier spectral Galerkin (FSG) method for
the fractional Zakharov--Kuznetsov (fZK) equation posed on a two-dimensional
periodic domain. The equation generalizes the classical ZK model by replacing
the Laplacian with a fractional Laplacian of order \(\alpha\in(0,2]\), thereby
covering the classical ZK equation \(\alpha=2\), the higher-dimensional
Benjamin--Ono--ZK equation \(\alpha=1\), and weaker fractional-dispersion
regimes \(0<\alpha<1\). We first propose a semi-discrete FSG scheme in space
that preserves the discrete analogues of mass, momentum, and Hamiltonian
energy. Using periodic Kato--Ponce product and commutator estimates, we prove
local-in-time uniform Sobolev bounds and strong convergence of the
semi-discrete approximations to the unique strong solution in
\(C([0,\bar T];L^2_{\mathrm{per}}(\Omega))\), for the initial 
condition in \(H^s_{\mathrm{per}}(\Omega)\), \(s\geq 2+\alpha\), and, as by product, we show that the existence and uniqueness of fZK equation in \(L^\infty(0,\bar T;H^s_{\mathrm{per}}(\Omega))
    \cap W^{1,\infty}(0,\bar T;L^2_{\mathrm{per}}(\Omega))\). We then introduce a modified projection adapted to the fractional transport-dispersive operator and prove optimal spatial error
estimates of order \(\mathcal O(N^{-r})\) for \(r>2+\alpha\), together with
exponential convergence for analytic solutions. An integrating-factor fourth-order four-stage Runge--Kutta time discretization is used to integrate the stiff fractional dispersive part exactly, and a fourth-order temporal error estimate is obtained under a high-regularity nonlinear stability assumption. Numerical experiments illustrate the accuracy, fractional-order dependence, and fully discrete conservation drift of the
method.
\end{abstract}

\maketitle

\section{Introduction}

The study of nonlinear dispersive equations has been a cornerstone of mathematical physics since the 19th century, with the Korteweg-de Vries (KdV) equation \cite{KortewegDeVries1895} serving as a fundamental model for unidirectional wave propagation in shallow water channels. The KdV equation exhibits a remarkable balance between nonlinear steepening and dispersive spreading, leading to the existence of stable solitary wave solutions \cite{Zabusky1965}. Its mathematical theory was firmly established through the pioneering work of \cite{BonaSmith1975} on the initial value problem and the discovery of its complete integrability in \cite{Gardner1967}.

While the KdV equation successfully describes one-dimensional wave phenomena, many physical systems require multidimensional models to capture transverse effects. The Zakharov-Kuznetsov (ZK) equation, originally derived in \cite{ZakharovKuznetsov1974} to describe ion-acoustic waves in magnetized plasma, provides a multidimensional generalization of the $1+1$ dimensional KdV equation:
\begin{equation}\label{zkeqn}
    \partial_t u + \partial_{x_1}\Delta u + u\partial_{x_1} u = 0,
\end{equation}
where $u = u(\mathbf{x},t)$, $\mathbf{x} = (x_1, \dots, x_n) \in \mathbb{R}^n$, $n \geq 2$. The physical relevance of the ZK equation extends beyond plasma physics to include applications in solid state physics \cite{kivshar1989dynamics} and fluid dynamics \cite{melkonian1989two}. A rigorous derivation from the Euler--Poisson system was provided in \cite{LannesLinaresSaut2013}, cementing its importance as a canonical multidimensional dispersive equation.

To incorporate nonlocal dispersion and to model sub-dispersive phenomena, one generalizes the classical ZK equation by replacing the Laplacian $\Delta$ with the fractional Laplacian $-(-\Delta)^{\frac{\alpha}{2}}$ of order $\alpha\in(0,2]$. The resulting fractional Zakharov--Kuznetsov (fZK) equation reads:
\[
    \partial_t u-\partial_{x_1}(-\Delta)^{\frac{\alpha}{2}}u
    +u\partial_{x_1}u=0.
\]
The full mathematical family \(0<\alpha\leq2\) has been studied from the point
of view of local smoothing and propagation of regularity; in particular,
Mendez \cite{Mendez2023kato} established a Kato-type smoothing effect of order
\(\alpha/2\) for the whole-space fZK equation. The case \(\alpha=2\) is the
classical ZK equation, while \(\alpha=1\) corresponds to a
higher-dimensional Benjamin--Ono--ZK type model. The range \(0<\alpha<1\)
represents a weaker fractional-dispersion regime in which the transport
nonlinearity is comparatively stronger than the dispersive smoothing.

In this work, we consider the initial value problem for the fZK equation posed on the two-dimensional torus, i.e., $\Omega=[-\pi,\pi]^2$ with periodic boundary conditions:
\begin{equation}\label{eqn:fZK}
    \begin{cases}
        \partial_t u - \partial_{x_1}(-\Delta)^{\alpha/2} u + u\partial_{x_1} u = 0, & \mathbf{x} = (x_1,x_2) \in \Omega, \;\; t \in (0,T], \\
        u(\mathbf{x},0) = u_0(\mathbf{x}), & \mathbf{x} \in \Omega,
    \end{cases}
\end{equation}
where $T > 0$ is a given final time and \(\alpha\in(0,2]\). The fractional Laplacian $(-\Delta)^{\alpha/2}$ is defined spectrally for periodic functions through the Fourier series expansion:
\begin{equation}
    (-\Delta)^{\alpha/2} f(\mathbf{x}) = \sum_{\mathbf{k} \in \mathbb{Z}^2} |\mathbf{k}|^{\alpha} \hat{f}(\mathbf{k}) e^{i\mathbf{k}\cdot\mathbf{x}},
\end{equation}
where $\hat{f}(\mathbf{k}) = (2\pi)^{-2} \int_{\Omega} f(\mathbf{x}) e^{-i\mathbf{k}\cdot\mathbf{x}}  d\mathbf{x}$ are the Fourier coefficients of $f$. The fZK equation \eqref{eqn:fZK} possesses several fundamental conserved quantities that play a crucial role in its analytical and numerical analysis. For sufficiently smooth solutions, the following invariants are preserved \cite{Mendez2023kato}:
\begin{equation}\label{convquant}
\begin{aligned}
  \mathcal{I}(u) &= \int_{\Omega} u(\mathbf{x}, t)\,d\mathbf{x}, &&(\text{mass}), \\
  \mathcal{M}(u) &= \int_{\Omega} u^2(\mathbf{x}, t)\,d\mathbf{x}, &&(\text{momentum}),\\
  \mathcal{H}(u) &= \frac{1}{2} \int_{\Omega} \left|(-\Delta)^{\frac{\alpha}{4}} u(\mathbf{x},t)\right|^2\,d\mathbf{x}
  - \frac{1}{6} \int_{\Omega} u^3(\mathbf{x},t)\,d\mathbf{x}, &&(\text{Hamiltonian energy}).
\end{aligned}
\end{equation}
These quantities are the periodic analogues of the whole-space conservation laws and are preserved by sufficiently smooth periodic solutions of \eqref{eqn:fZK}. The conservation laws \eqref{convquant} are essential for understanding the long-time behavior of solutions and play a fundamental role in establishing well-posedness results \cite{Faminski1995,Mendez2023kato}.

The well-posedness theory for the classical ZK equation ($\alpha=2$) has been extensively developed. Faminski  \cite{Faminski1995} established global well-posedness in the energy space $H^1(\mathbb{R}^2)$, while Linares et al. \cite{LinaresPastor2009} proved local well-posedness in $H^s(\mathbb{R}^2)$ for $s > 3/4$ using the Fourier restriction method. Subsequent refinements in \cite{MolinetPilod2015} and \cite{GrunrockHerr2014} extended this to $H^s(\mathbb{R}^2)$ for $s > 1/2$. In the periodic setting, \cite{MolinetPilod2015} established local well-posedness in $H^s(\mathbb{T}^2)$ for $s > 1$. The three-dimensional case was treated in \cite{RibaudVento2012}, who proved local well-posedness in $H^s(\mathbb{R}^3)$ for $s > 1$. Recent breakthroughs in \cite{HerrKinoshita2021} and \cite{Kinoshita2021} established well-posedness in the optimal Sobolev ranges $H^s(\mathbb{R}^2)$ for $s > -1/4$ and $H^s(\mathbb{R}^n)$ for $s > (n-4)/2$ when $n \geq 3$.

The fZK equation with $\alpha=1$ can be reformulated as
\begin{equation}
    \partial_t u - \mathcal{R}_1 \Delta u + u\partial_{x_1} u = 0,
\end{equation}
where $\mathcal{R}_1$ denotes the Riesz transform in the $x_1$-direction, defined by
\begin{equation}
    \mathcal{R}_1 f(x) = \frac{\Gamma\left(\frac{n+1}{2}\right)}{\pi^{\frac{n+1}{2}}} \text{p.v.} \int_{\mathbb{R}^n} \frac{x_1 - y_1}{|x-y|^{n+1}} f(y) dy,
\end{equation}
where \text{p.v.} stands for the principal value.
This equation was first derived in \cite{Shrira1989} to model bi-dimensional long-wave perturbations in boundary-layer shear flows and represents a higher-dimensional generalization of the Benjamin-Ono equation \cite{benjamin1967internal,ono1975algebraic}
\begin{equation}
    \partial_t u - \mathcal{H}\partial_x^2 u + u\partial_x u = 0,
\end{equation}
where $\mathcal{H}$ is the Hilbert transform \cite{Tao2004}. The well-posedness theory for the $\alpha=1$ case was initiated in \cite{Hickman2019}, who established local well-posedness in $H^s(\mathbb{R}^n)$ for $s > 3/5$ in two dimensions and $s > n/2 + 1/2$ for $n \geq 3$. Schippa \cite{Schippa2020} improved these results to $H^s(\mathbb{R}^n)$ for $s > (n+3)/2 - \alpha$ with $n \geq 2$ and $\alpha \in [1,2)$. The periodic case was also addressed in \cite{Schippa2020}, establishing local well-posedness in $H^s(\mathbb{T}^n)$.

For the one-dimensional Benjamin--Ono equation itself, the well-posedness theory has seen remarkable progress. I\'orio \cite{Iorio1986} established local well-posedness in $H^{3/2}(\mathbb{R})$, while Tao \cite{Tao2004} proved global well-posedness in $H^1(\mathbb{R})$. The groundbreaking work \cite{killip2024sharp} established sharp global well-posedness in $H^{s}(\R)$ for $s>-\frac{1}{2}$, resolving a long-standing open problem. These developments have significantly influenced the study of related nonlocal dispersive equations.

For intermediate values $\alpha \in (0,2)/\{1\}$, the analytical theory remains less developed. Mendez \cite{Mendez2023kato} made significant progress by establishing a Kato-type smoothing effect, showing that solutions gain $\alpha/2$ derivatives locally in both space and time. This smoothing property, first discovered in \cite{Kato1983} for the KdV equation and later generalized in \cite{ConstantinSaut1988}, plays a crucial role in the analysis of dispersive equations. Additional regularity properties and smoothing effects for the solution of the related nonlocal equations were studied in \cite{GinibreVelo1989, GinibreVelo1991,KenigPonceVega1991, galtung2018convergent, dwivedi2024stability}.

Despite substantial analytical progress, numerical methods for the fZK equation \eqref{eqn:fZK} remain largely unexplored. For the classical ZK equation \eqref{zkeqn}, various numerical methods have been proposed, including multi-symplectic spectral discretizations in \cite{BridgesReich2001,Qian2012, li2015multi} and stable discontinuous Galerkin methods in \cite{xu2005local, SunLiuWei2016}. However, these works lack rigorous error analysis with nonlinear flux, convergence of approximate solutions to the unique solution of equation \eqref{zkeqn}, and spectral convergence. To the best of our knowledge, no numerical schemes with complete convergence analysis exist for the fZK equation \eqref{eqn:fZK}. Fourier spectral Galerkin (FSG) methods are particularly effective for periodic fractional and nonlocal PDEs because fractional powers of the Laplacian are diagonal in the Fourier basis, allowing accurate and efficient representation of nonlocal dispersion. This has been successfully used in the analysis and approximation of fractional and nonlocal evolution equations, including
fractional Cahn--Hilliard type models \cite{AinsworthMao2017SINA,AinsworthMao2017CSF}, fractional KdV
models \cite{dwivedi2024numerical}, and fractional Camassa--Holm equations \cite{dwivedi2025convergent}. On the time-discretization side,
integrating-factor and exponential Runge--Kutta methods are standard tools for
stiff dispersive spectral systems because they treat the high-order linear
oscillatory part exactly while keeping the nonlinear part explicit
\cite{CoxMatthews2002,HochbruckOstermann2010,KassamTrefethen2005,vaissmoradi2009error}.
These advantages motivate the integrating-factor fourth-order four-stage Runge--Kutta (IFRK4)--FSG framework developed in this paper.

This work presents a novel Fourier spectral Galerkin framework for the fZK equation \eqref{eqn:fZK}, providing the first complete numerical analysis for this model and thereby bridging a critical gap in the existing literature. The state-of-the-art methodology capitalizes on the inherent affinity between spectral techniques and periodic boundary conditions, combined with the computationally efficient representation of fractional operators within a Fourier basis \cite{canuto2006spectral, Carleson}. The main contributions of this paper are:
\begin{itemize}
    \item Formulation of a semi-discrete spectral scheme that exactly preserves the discrete analogs of the continuous mass, momentum, and energy invariants. This structure-preserving property ensures the long-term stability and physical fidelity of numerical simulations.
    
    \item Rigorous proof of existence and uniqueness for the semi-discrete solution. This is established through the application of standard ordinary differential equation theory to the Fourier coefficient system.
    
   \item Strong convergence analysis of the spectral approximations to the exact solution for the full fractional range \(0<\alpha\leq2\). Using uniform Sobolev bounds in \(H^s_{\mathrm{per}}(\Omega)\), \(s\geq 2+\alpha\), Rellich compactness, and the Arzelà--Ascoli theorem in \(L^2_{\mathrm{per}}(\Omega)\), we prove convergence to the unique local strong solution of the fZK equation \eqref{eqn:fZK}. The result extends to any time interval on which an independent \(H^s\)-bound for the exact solution is available.
    
    \item Derivation of optimal spatial error estimates through a modified projection operator adapted to the fractional transport-dispersive structure. A refined adjoint estimate allows the projection argument to remain uniform in the Fourier cut-off for all \(0<\alpha\leq2\). This yields an optimal \(\mathcal O(N^{-r})\) estimate for \(r>2+\alpha\), and exponential convergence for analytic solutions.
    
    \item Development of an IFRK4 time discretization that exactly integrates the stiff fractional dispersive part. We give the correct local truncation error formulation and a fourth-order temporal error estimate under a high-regularity nonlinear stability assumption for the semi-discrete and fully discrete trajectories.
    
    \item Comprehensive numerical investigation validating the theoretical results and demonstrating the method's performance across various fractional orders. The experiments include convergence verification, solitary wave propagation, and long-time behavior analysis.
\end{itemize}

The paper is organized as follows: Section \ref{sec2} introduces the functional framework and fractional operators. Section \ref{sec3} presents the Fourier-Galerkin scheme and establishes well-posedness. Section \ref{sec4} develops an optimal error analysis. Section \ref{sec5} describes the time discretization, and Section \ref{sec6} presents numerical experiments validating our theoretical results.

\section{Preliminaries: Function Spaces and Fractional Operators}
\label{sec2}

This section establishes the functional framework and analytical tools required for our analysis of the fZK equation \eqref{eqn:fZK}. We consider the problem on the two-dimensional torus $\Omega = [-\pi,\pi]^2$ with periodic boundary conditions. The space of square-integrable periodic functions is denoted by $L^2_{\mathrm{per}}(\Omega)$, equipped with the inner product and norm
\[
(u,v) = \int_{\Omega} u(\mathbf{x})\overline{v(\mathbf{x})}\,d\mathbf{x}, \quad \|u\| = (u,u)^{1/2}.
\]
For any real number $s \geq 0$, the periodic Sobolev space $H^s_{\mathrm{per}}(\Omega)$ consists of all functions $u$ whose Fourier coefficients $\hat{u}(\mathbf{k})$ satisfy
\[
\|u\|_s^2 = \sum_{\mathbf{k} \in \mathbb{Z}^2} (1 + |\mathbf{k}|^2)^s |\hat{u}(\mathbf{k})|^2 < \infty,
\]
where $|\mathbf{k}|^2 = k_1^2 + k_2^2$ and  every function $u \in H^s_{\mathrm{per}}(\Omega)$ admits a Fourier series expansion
\[
u(\mathbf{x}) = \sum_{\mathbf{k} \in \mathbb{Z}^2} \hat{u}(\mathbf{k}) e^{i\mathbf{k} \cdot \mathbf{x}}.
\] 
Note that $H^0_{\mathrm{per}}(\Omega)$ coincides with $L^2_{\mathrm{per}}(\Omega)$. For numerical approximation, we introduce the finite-dimensional space of trigonometric polynomials
\[
\mathbb{X}_N = \mathrm{span}\left\{ e^{i\mathbf{k} \cdot \mathbf{x}} : |\mathbf{k}|_{\infty} = \max\{|k_1|, |k_2|\} \leq N \right\}.
\]
The orthogonal projection $\Pi_N: L^2_{\mathrm{per}}(\Omega) \to \mathbb{X}_N$ is defined by
\[
\Pi_N u(\mathbf{x}) = \sum_{|\mathbf{k}|_{\infty} \leq N} \hat{u}(\mathbf{k}) e^{i\mathbf{k} \cdot \mathbf{x}},
\]
which satisfies the orthogonality condition
\begin{equation}
\label{eq:proj_ortho}
(\Pi_N u - u, \psi) = 0 \quad \text{for all } \psi \in \mathbb{X}_N.
\end{equation}
Since trigonometric polynomials are dense in $L^2_{\mathrm{per}}(\Omega)$, the orthogonal Fourier projection satisfies
\[
    \|\Pi_N u-u\|\to 0 \qquad \text{as } N\to\infty
\]
for every $u\in L^2_{\mathrm{per}}(\Omega)$, see \cite[Theo. 13.2]{hesthaven2017numerical}. More generally, if
$u\in H^r_{\mathrm{per}}(\Omega)$ with $r\geq 0$, then
\[
    \|\Pi_N u\|_r\leq \|u\|_r,
    \qquad
    \|\Pi_N u-u\|_r\to 0
    \qquad \text{as } N\to\infty .
\]
For later use, we also introduce the Bessel potential operator
\[
    \Lambda^s := (I-\Delta)^{s/2}, \qquad
    \Lambda^s f(\mathbf{x})
    =
    \sum_{\mathbf{k}\in\mathbb{Z}^2}
    (1+|\mathbf{k}|^2)^{s/2}\hat f(\mathbf{k})
    e^{i\mathbf{k}\cdot\mathbf{x}}.
\]
Thus $\|f\|_s=\|\Lambda^s f\|$. In two space dimensions, we shall use the standard Sobolev embeddings
\(H^s_{\mathrm{per}}(\Omega)\hookrightarrow L^\infty(\Omega)\)
    for  \(s>1\) and \( H^s_{\mathrm{per}}(\Omega)\hookrightarrow W^{1,\infty}(\Omega)\) for \(s>2\).
We now establish several fundamental properties of the fractional Laplacian that will be essential for our analysis.
\begin{lemma}
\label{lem:frac_properties}
The fractional Laplacian $(-\Delta)^{\alpha/2}$ satisfies the following properties:
\begin{enumerate}[label=(\roman*)]
    \item For $f, g \in H^{\alpha}_{\mathrm{per}}(\Omega)$, $\alpha \geq 0$, we have
    \begin{equation}
    \label{eq:frac_symmetry}
    \left((-\Delta)^{\alpha/2}f, g\right) = \left(f, (-\Delta)^{\alpha/2}g\right),
    \end{equation}
    and for real-valued $f\in H^{1+\alpha}_{\mathrm{per}}(\Omega)$,
    \begin{equation}
    \label{eq:frac_orthogonality}
    \left((-\Delta)^{\alpha/2}(\partial_{x_1}f), f\right) = 0.
    \end{equation}
    \item For $\alpha_1, \alpha_2 \geq 0$ and $f, g \in H^{\alpha_1 + \alpha_2}_{\mathrm{per}}(\Omega)$, we have
    \begin{equation}
    \label{eq:semigroup_property}
    \left((-\Delta)^{(\alpha_1 + \alpha_2)/2}f, g\right) = \left((-\Delta)^{\alpha_1/2}f, (-\Delta)^{\alpha_2/2}g\right),
    \end{equation}
    and equivalently,
    \begin{equation}
    \label{eq:commutative_property}
    (-\Delta)^{(\alpha_1 + \alpha_2)/2}f = (-\Delta)^{\alpha_1/2}(-\Delta)^{\alpha_2/2}f = (-\Delta)^{\alpha_2/2}(-\Delta)^{\alpha_1/2}f.
    \end{equation}
    
    \item For the orthogonal projection $\Pi_N$ and $f \in H^r_{\mathrm{per}}(\Omega)$ with $r \geq \alpha \geq 0$, the fractional Laplacian commutes with the projection:
    \begin{equation}
    \label{eq:projection_commute}
    (-\Delta)^{\alpha/2}(\Pi_N f) = \Pi_N((-\Delta)^{\alpha/2}f).
    \end{equation}
\end{enumerate}
\end{lemma}

\begin{proof}
Property \eqref{eq:frac_symmetry} follows directly from Parseval's identity and from the fact that $f$ and $g$ are real valued
\[
\left((-\Delta)^{\alpha/2}f, g\right) = \sum_{\mathbf{k} \in \mathbb{Z}^2} |\mathbf{k}|^{\alpha} \hat{f}(\mathbf{k}) \overline{\hat{g}(\mathbf{k})} = \left(f, (-\Delta)^{\alpha/2}g\right).
\]
Periodicity of $f$, integration by parts, the commutative property of $(-\Delta)^{\alpha/2}$ with partial derivatives, and \eqref{eq:frac_symmetry} imply \eqref{eq:frac_orthogonality}.
Property \eqref{eq:semigroup_property} follows from the identity
\[
\left((-\Delta)^{(\alpha_1 + \alpha_2)/2}f, g\right) = \sum_{\mathbf{k} \in \mathbb{Z}^2} |\mathbf{k}|^{\alpha_1 + \alpha_2} \hat{f}(\mathbf{k}) \overline{\hat{g}(\mathbf{k})} = \left((-\Delta)^{\alpha_1/2}f, (-\Delta)^{\alpha_2/2}g\right).
\]
The commutative property \eqref{eq:commutative_property} is immediate from the commutativity of the Fourier multipliers $|\mathbf{k}|^{\alpha_1}$ and $|\mathbf{k}|^{\alpha_2}$. Finally, property \eqref{eq:projection_commute} holds because both operators act as Fourier multipliers, and $\Pi_N$ truncates high frequencies
\[
(-\Delta)^{\alpha/2}(\Pi_N f)(\mathbf{x}) = \sum_{|\mathbf{k}|_{\infty} \leq N} |\mathbf{k}|^{\alpha} \hat{f}(\mathbf{k}) e^{i\mathbf{k} \cdot \mathbf{x}} = \Pi_N((-\Delta)^{\alpha/2}f)(\mathbf{x}).
\]
This completes the proof.
\end{proof}

Next, we recall the product and commutator estimates that will be used to
control the nonlinear term. These estimates are norm estimates; no pointwise
fractional Leibniz rule is used.

\begin{lemma}[Fractional product and commutator estimates]
\label{lem:product_estimate}
Let $s>0$. There exists a constant $C_s>0$, depending only on $s$ and
$\Omega$, such that the following estimates hold.

\begin{enumerate}[label=(\roman*)]
    \item If $f,g\in H^s_{\mathrm{per}}(\Omega)\cap L^\infty(\Omega)$, then
    \begin{equation}
    \label{eq:KP_product}
        \|\Lambda^s(fg)\|
        \leq C_s\left(
        \|f\|_{L^\infty(\Omega)}\|g\|_s
        +
        \|g\|_{L^\infty(\Omega)}\|f\|_s
        \right).
    \end{equation}

    \item If $f,g\in H^s_{\mathrm{per}}(\Omega)\cap W^{1,\infty}(\Omega)$, then
    for $j=1,2$,
    \begin{equation}
    \label{eq:KP_commutator}
        \|[\Lambda^s,f]\partial_{x_j}g\|
        \leq C_s\left(
        \|\nabla f\|_{L^\infty(\Omega)}\|g\|_s
        +
        \|f\|_s\|\partial_{x_j}g\|_{L^\infty(\Omega)}
        \right),
    \end{equation}
    where
    \begin{equation}\label{eq:comm_decomp}
        [\Lambda^s,f]h:=\Lambda^s(fh)-f\Lambda^s h .
    \end{equation}
\end{enumerate}

In particular, if $s>2$, then
\begin{equation}
\label{eq:KP_diag}
    \|[\Lambda^s,f]\partial_{x_j}f\|
    \leq C_s \|f\|_s^2,
    \qquad f\in H^s_{\mathrm{per}}(\Omega).
\end{equation}
\end{lemma}

\begin{proof}
For smooth periodic functions, the product estimate \eqref{eq:KP_product} is the \(L^2\)-Sobolev form of the periodic fractional Leibniz rule on the torus; see \cite[Prop. 1 and (1.5)]{BenyiOhZhao2025}. For the corresponding Euclidean Kato--Ponce product inequality, see
\cite[Theo. 1]{GrafakosOh2014}. In the present $L^2$ form, it is obtained by applying the periodic Kato--Ponce inequality to $\Lambda^s(fg)$.

The commutator estimate \eqref{eq:KP_commutator} is the corresponding
Kato--Ponce commutator estimate; see \cite{GrafakosOh2014,KatoPonce1988,KenigPonceVega1993}.
Applied with $h=\partial_{x_j}g$, it gives
\[
    \|[\Lambda^s,f]\partial_{x_j}g\|
    \leq C_s\left(
    \|\nabla f\|_{L^\infty(\Omega)}
    \|\Lambda^{s-1}\partial_{x_j}g\|
    +
    \|\Lambda^s f\|\|\partial_{x_j}g\|_{L^\infty(\Omega)}
    \right).
\]
Since
\[
    \|\Lambda^{s-1}\partial_{x_j}g\|\leq \|g\|_s,
    \qquad
    \|\Lambda^s f\|=\|f\|_s,
\]
we obtain \eqref{eq:KP_commutator}. The extension from smooth periodic
functions to the stated Sobolev classes follows by standard density.

Finally, if $s>2$, then
$H^s_{\mathrm{per}}(\Omega)\hookrightarrow W^{1,\infty}(\Omega)$, and
\eqref{eq:KP_diag} follows immediately from \eqref{eq:KP_commutator}.
\end{proof}

Finally, we state some useful inequalities for trigonometric polynomials in $\mathbb{X}_N$.

\begin{lemma}[Inverse and Sobolev Inequalities]
\label{lem:inverse_inequalities}
For any $v_N \in \mathbb{X}_N$, the following estimates hold:
\begin{enumerate}[label=(\roman*)]
    \item Inverse inequality:
    \[
    \|\partial_{x_j} v_N\| \leq N \|v_N\|, \quad j = 1, 2.
    \]
    \item Sobolev inequality:
    \[
    \|v_N\|_{L^\infty(\Omega)} \leq 3 N \|v_N\|.
    \]
    \item For any $r \geq 0$, we have the embedding
    \[
    \|v_N\| \leq \|v_N\|_r.
    \]
\end{enumerate}
\end{lemma}
\begin{proof}
The inverse inequality follows from Parseval's identity and the fact that $|k_j| \leq N$ for $|\mathbf{k}|_{\infty} \leq N$:
\[
\|\partial_{x_j} v_N\|^2 = \sum_{|\mathbf{k}|_{\infty} \leq N} |k_j|^2 |\hat{v}_N(\mathbf{k})|^2 \leq N^2 \sum_{|\mathbf{k}|_{\infty} \leq N} |\hat{v}_N(\mathbf{k})|^2 = N^2 \|v_N\|^2.
\]
The Sobolev inequality is proved using the Cauchy-Schwarz inequality and the fact that there are $O(N^2)$ Fourier modes:
\[
|v_N(\mathbf{x})| \leq \sum_{|\mathbf{k}|_{\infty} \leq N} |\hat{v}_N(\mathbf{k})| \leq \left( \sum_{|\mathbf{k}|_{\infty} \leq N} 1 \right)^{1/2} \left( \sum_{|\mathbf{k}|_{\infty} \leq N} |\hat{v}_N(\mathbf{k})|^2 \right)^{1/2} \leq 3 N \|v_N\|.
\]
The embedding result is immediate from the definition of the Sobolev norm.
\end{proof}

\section{Fourier-Galerkin scheme and well-posedness}\label{sec3}
We consider the fractional Zakharov-Kuznetsov (fZK) equation \eqref{eqn:fZK} on the periodic domain $\Omega$. Multiplying by a test function $v \in H^{\alpha}_{\text{per}}(\Omega)$ and integrating by parts yields the variational form. Find $u \in H^{\alpha}_{\text{per}}(\Omega)$ such that
\begin{equation}
(\partial_t u, v) + ((-\Delta)^{\frac{\alpha}{2}} u, \partial_{x_1} v) - \frac{1}{2}(u^2, \partial_{x_1} v) = 0, \quad \forall \,v \in H^{1+\alpha}_{\text{per}}(\Omega).
\end{equation}
  The Fourier spectral-Galerkin approximation consists of finding $u_N \in \mathbb{X}_N$ such that
\begin{equation}\label{fsg:fZK}
(\partial_t u_N, v) + ((-\Delta)^{\frac{\alpha}{2}} u_N, \partial_{x_1} v) - \frac{1}{2}(u_N^2, \partial_{x_1} v) = 0 \quad \forall v \in \mathbb{X}_N,
\end{equation}
with the initial condition $(u_N(0),v) = (u_0,v)$.

Let $
u_N(\mathbf{x}, t) = \sum_{|\mathbf{k}|_\infty \leq N} \hat{u}_N(\mathbf{k}, t) e^{i \mathbf{k} \cdot \mathbf{x}}$
and taking test functions $v(\mathbf{x}) = e^{i \mathbf{m} \cdot \mathbf{x}}$ for $|\mathbf{m}|_\infty \leq N$, we obtain the following system of ordinary differential equations for the Fourier coefficients:
\begin{equation}
\frac{d}{dt} \hat{u}_N(\mathbf{m}, t) = i m_1 |\mathbf{m}|^\alpha \hat{u}_N(\mathbf{m}, t) - \frac{1}{2} i m_1 \sum_{\substack{|\mathbf{k}|_\infty \leq N \\ |\mathbf{m} - \mathbf{k}|_\infty \leq N}} \hat{u}_N(\mathbf{k}, t) \hat{u}_N(\mathbf{m} - \mathbf{k}, t), \quad |\mathbf{m}|_\infty \leq N.
\end{equation}
Let $M := (2N+1)^2$ and let $U(t)$ denote the vector of all Fourier coefficients $\hat{u}_N(\mathbf{m}, t)$ for $|\mathbf{m}|_\infty \leq N$. Then the above system can be written as
\begin{equation}\label{ode-sys}
\frac{dU}{dt}(t) = G(U(t)),
\end{equation}
where $G: \mathbb{R}^M \to \mathbb{R}^M$ is defined by
\begin{equation*}
G(U(t)) = \{ G_j(U(t)) \}_{1\leq j \leq M},
\end{equation*}
and each component function 
 $G_{j}: \mathbb{R}^M \to \mathbb{R}$ ($1\leq j \leq M$) are smooth and locally Lipschitz continuous.
The initial condition $u_N(0) = \Pi_N u_0$ implies
\begin{equation}\label{ini-data}
U(0)
=
\left\{\hat u_N(\mathbf m,0)\right\}_{|\mathbf m|_\infty\leq N} = \left\{\hat u_0(\mathbf m)\right\}_{|\mathbf m|_\infty\leq N}.
\end{equation}
By the Picard-Lindelöf theorem, the system \eqref{ode-sys} with initial condition \eqref{ini-data} has a unique solution $U(t) \in C^1([0, T_{\text{max}}); \mathbb{R}^M)$ for some $T_{\text{max}}>0$. To uniquely extend the solution for all $t > 0$, we will show that the $L^2$ norm of $u_N$ is conserved.

\begin{theorem}\label{conslemma}
The scheme \eqref{fsg:fZK} conserves mass, momentum, and the Hamiltonian specified by \eqref{convquant}. In particular, the approximation $u_N \in \mathbb{X}_N$ obtained by the scheme \eqref{fsg:fZK} preserves the following invariants
\begin{align}
& \frac{d}{dt}\mathcal{I}(u_N)  = 0, \label{cons1} \\
 &\frac{d}{dt}\mathcal M(u_N) = 0, \label{cons2}\\
 &\frac{d}{dt}\mathcal  H(u_N) = 0. \label{cons3}
\end{align}
\end{theorem}

\begin{proof}
Mass conservation \eqref{cons1} follows from taking $v = 1$ in \eqref{fsg:fZK}. Now take $v = u_N$ in \eqref{fsg:fZK}, we have
\[
(\partial_t u_N, u_N) + ((-\Delta)^{\frac{\alpha}{2}} u_N, \partial_{x_1} u_N) - \frac{1}{2}(u_N^2, \partial_{x_1} u_N) = 0.
\]
The second term vanishes due to the property \eqref{eq:frac_orthogonality} and the third term vanishes due to periodicity:
\[
(u_N^2, \partial_{x_1} u_N) = \frac{1}{3} \int_{\Omega} \partial_{x_1} (u_N^3)  d\mathbf{x} = 0.
\]
This implies \eqref{cons2}. 
Now we take $v = \Pi_N \left( (-\Delta)^{\frac{\alpha}{2}} u_N - \frac{1}{2} u_N^2 \right) \in \mathbb{X}_N$ in \eqref{fsg:fZK}, we have 
\[
\left(\partial_t u_N, \Pi_N \left( (-\Delta)^{\frac{\alpha}{2}} u_N - \frac{1}{2} u_N^2 \right)\right) + \left((-\Delta)^{\frac{\alpha}{2}} u_N - \frac{1}{2}u_N^2, \partial_{x_1} \Pi_N \left( (-\Delta)^{\frac{\alpha}{2}} u_N - \frac{1}{2} u_N^2 \right)\right)= 0.
\]
Utilizing the projection properties, we have 
\[
\left(\partial_t u_N,   (-\Delta)^{\frac{\alpha}{2}} u_N - \frac{1}{2} u_N^2 \right) +\left(\Pi_N\left((-\Delta)^{\frac{\alpha}{2}} u_N - \frac{1}{2}u_N^2\right), \partial_{x_1} \Pi_N \left( (-\Delta)^{\frac{\alpha}{2}} u_N - \frac{1}{2} u_N^2 \right)\right)= 0.
\]
The second term above vanishes due to the periodicity of $u_N$ and property \eqref{eq:frac_symmetry} implies 
\[
\left(\partial_t u_N,   (-\Delta)^{\frac{\alpha}{2}} u_N - \frac{1}{2} u_N^2 \right) = \frac{d}{dt}\left(\frac{1}{2}((-\Delta)^{\frac{\alpha}{4}} u_N)^2- \frac{1}{6}u_N^3,  1 \right) =0,
\]
which implies \eqref{cons3}.
\end{proof}

Now we provide the local existence and uniqueness result for the fZK equation \eqref{eqn:fZK} in the regularity class needed for the Galerkin convergence argument.

\begin{theorem}\label{thm:exisuq}
Let \(0<\alpha\leq2\), and let \(s\geq 2+\alpha\). Assume that
\(u_0\in H^s_{\mathrm{per}}(\Omega)\). Then there exists a time
\(\bar T>0\), depending on \(\|u_0\|_s\), such that the fZK equation
\eqref{eqn:fZK} has a unique strong solution
\[
    u\in L^\infty(0,\bar T;H^s_{\mathrm{per}}(\Omega))
    \cap W^{1,\infty}(0,\bar T;L^2_{\mathrm{per}}(\Omega)).
\]
Moreover, the semi-discrete Fourier--Galerkin approximations \(u_N\)
generated by \eqref{fsg:fZK} converge to \(u\) strongly in
\(C([0,\bar T];L^2_{\mathrm{per}}(\Omega))\).
\end{theorem}

To prove the above theorem, we first state and prove the following convergence lemma. Let us first introduce the definition of a weak solution for the fZK equation \eqref{eqn:fZK}. A function
$u \in L^\infty([0,T];L^2_{\mathrm{per}}(\Omega))$ is said to be a weak
solution of \eqref{eqn:fZK} with initial data $u_0 \in L^2_{\mathrm{per}}(\Omega)$ if, for every test function $\varphi \in C_c^\infty([0,T)\times\Omega)$, the following variational
formulation holds:
\begin{equation}\label{weaksoln}
    \int_0^T (u,\partial_t\varphi)\,dt
    -\int_0^T \left(u,(-\Delta)^{\frac{\alpha}{2}}\partial_{x_1}\varphi\right)\,dt
    +\frac{1}{2}\int_0^T \left(u^2,\partial_{x_1}\varphi\right)\,dt
    +(u_0,\varphi(0))=0.
\end{equation}

\begin{lemma}\label{conv_lemma}
Let \(0<\alpha\leq2\), and let \(s\geq 2+\alpha\). Assume
\(u_0\in H^s_{\mathrm{per}}(\Omega)\), and let \(u_N\) be the approximate
solution of the fZK equation \eqref{eqn:fZK} obtained by the scheme
\eqref{fsg:fZK} with \(u_N(0)=\Pi_Nu_0\). Then there exists a finite time
\(\bar T>0\), depending on \(\|u_0\|_s\) but independent of \(N\), such that
\begin{align}
    \label{bd1}\|u_N(\cdot,t)\| &\leq C,\\
    \label{bd2}\|\partial_tu_N(\cdot,t)\| &\leq C,\\
    \label{bd3}\|u_N(\cdot,t)\|_s &\leq C,
\end{align}
for all \(0\leq t\leq\bar T\), where
\(C=C(\|u_0\|_s,\alpha,s,\bar T)\) is independent of \(N\). Moreover, there exists a subsequence, still denoted by $u_N$,
and a function
\[
    \bar u\in L^\infty(0,\bar T;H^s_{\mathrm{per}}(\Omega))
    \cap W^{1,\infty}(0,\bar T;L^2_{\mathrm{per}}(\Omega)),
\]
such that
\[
    u_N\to \bar u
    \qquad\text{strongly in }
    C([0,\bar T];L^2_{\mathrm{per}}(\Omega)).
\]
The limit $\bar u$ is the unique strong solution of \eqref{eqn:fZK} on
$[0,\bar T]$. Consequently, the whole sequence $u_N$ converges to $\bar u$ in $C([0,\bar T];L^2_{\mathrm{per}}(\Omega))$.
\end{lemma}

\begin{proof}
The $L^2$ bound \eqref{bd1} follows from the conservation of momentum in
Theorem \ref{conslemma}. We now derive the uniform higher Sobolev estimate. Since $u_N\in\mathbb X_N$
and $\Lambda^{2s}u_N\in\mathbb X_N$, we may take
$v=\Lambda^{2s}u_N$ in \eqref{fsg:fZK}. This gives
\[
    \frac12\frac{d}{dt}\|u_N\|_s^2
    +
    \left((-\Delta)^{\frac{\alpha}{2}}u_N,
    \partial_{x_1}\Lambda^{2s}u_N\right)
    +
    \left(u_N\partial_{x_1}u_N,\Lambda^{2s}u_N\right)
    =0.
\]
The dispersive term vanishes. Indeed, since $\Lambda^s$ commutes with
$\partial_{x_1}$ and $(-\Delta)^{\alpha/2}$,
\[
    \left((-\Delta)^{\frac{\alpha}{2}}u_N,
    \partial_{x_1}\Lambda^{2s}u_N\right)
    =
    \left((-\Delta)^{\frac{\alpha}{2}}\Lambda^s u_N,
    \partial_{x_1}\Lambda^s u_N\right)=0,
\]
by the skew-symmetry property in Lemma \ref{lem:frac_properties}. Hence
\[
    \frac12\frac{d}{dt}\|u_N\|_s^2 =- I_N,
    \qquad  I_N:= \left(\Lambda^s(u_N\partial_{x_1}u_N),\Lambda^s u_N\right).
\]
We estimate $I_N$ by a commutator decomposition \eqref{eq:comm_decomp}. Therefore
\[
    I_N
    =
    \left(u_N\partial_{x_1}\Lambda^s u_N,\Lambda^s u_N\right)
    +
    \left([\Lambda^s,u_N]\partial_{x_1}u_N,\Lambda^s u_N\right).
\]
For the first term, periodicity gives \(
    \left(u_N\partial_{x_1}\Lambda^s u_N,\Lambda^s u_N\right) = -\frac12
    \left(\partial_{x_1}u_N,|\Lambda^s u_N|^2\right)\), and hence
\[
    \left|
    \left(u_N\partial_{x_1}\Lambda^s u_N,\Lambda^s u_N\right)
    \right|
    \leq
    \frac12\|\partial_{x_1}u_N\|_{L^\infty(\Omega)}
    \|u_N\|_s^2 .
\]
For the commutator term, Lemma \ref{lem:product_estimate} gives
\[
    \|[\Lambda^s,u_N]\partial_{x_1}u_N\|
    \leq
    C_s\left(
    \|\nabla u_N\|_{L^\infty(\Omega)}\|u_N\|_s
    +
    \|u_N\|_s\|\partial_{x_1}u_N\|_{L^\infty(\Omega)}
    \right).
\]
Since $s>2$, the embedding
$H^s_{\mathrm{per}}(\Omega)\hookrightarrow W^{1,\infty}(\Omega)$ implies
\[
    \|\nabla u_N\|_{L^\infty(\Omega)}
    +
    \|\partial_{x_1}u_N\|_{L^\infty(\Omega)}
    \leq C\|u_N\|_s .
\]
Consequently,
\[
    |I_N|
    \leq C_s\|u_N\|_s^3 .
\]
Thus
\begin{equation}\label{diffine}
    \frac{d}{dt}\|u_N(t)\|_s^2
    \leq
    C_s\|u_N(t)\|_s^3 .
\end{equation}
Define $D(t) := \|u_N(t)\|_{s}^2$. From the differential inequality \eqref{diffine}, we get
\(
\frac{dD}{dt} \leq C D^{3/2},
\)
 with initial condition \(D(0) = D_0 := \|u_0\|_{s}^2\). Consider the equation
\[
\frac{de}{dt} = C e^{3/2}, \quad e(0) = D_0,
\]
where $C$ is independent of $t$. Solving by separation of variables yields
\[
e(t) = \left( \frac{1}{\sqrt{D_0}} - \frac{Ct}{2} \right)^{-2},
\]
which is increasing and convex for \(t < T_b := \frac{2}{C\sqrt{D_0}}\). Define
\(
\bar T := \frac{T_b}{2} = \frac{1}{C\sqrt{D_0}},
\)
which gives the bound
\[
e(t) \leq e(\bar T) = 4D_0, \quad \text{for all } t \leq \bar T.
\]
By standard comparison theory for differential inequalities, we conclude
\[
D(t) \leq e(t) \leq 4D_0 \quad \text{for } t \leq \bar T,
\]
and consequently
\[
\|u_N(t)\|_{s} \leq 2\|u_0\|_{s} \leq C,
\]
where the constant $C$ depends on the initial data. This proves \eqref{bd3} for all $t\leq \bar T$. Now take the test function $v = \partial_t u_N$ in \eqref{fsg:fZK}, which yields
\begin{equation*}
\norm{\partial_t u_N}^2 =  ((-\Delta)^{\frac{\alpha}{2}} \partial_{x_1}u_N, \partial_t u_N) - (u_N \partial_{x_1}u_N , \partial_t u_N).
\end{equation*}
Utilizing the Cauchy-Schwarz inequality, $s\geq2+\alpha$, and the bounds \eqref{bd1} and \eqref{bd3} with the Sobolev embedding, we have
\begin{equation*}
\norm{\partial_t u_N} \leq   \norm{(-\Delta)^{\frac{\alpha}{2}} \partial_{x_1}u_N} +  \norm{u_N}_{L^\infty(\Omega)} \norm{u_N}_{1+\alpha}  \leq C\|u_N\|_s+C\|u_N\|_s^2, \qquad \forall\, t\leq \bar T.
\end{equation*}
Together with \eqref{bd3}, this gives \eqref{bd2}. We next prove compactness. The estimate \eqref{bd2} gives equicontinuity
in $L^2_{\mathrm{per}}(\Omega)$:
\[
    \|u_N(t)-u_N(\tau)\|
    \leq
    \int_\tau^t \|\partial_tu_N(r)\|\,dr
    \leq C|t-\tau|,
    \qquad 0\leq \tau<t\leq \bar T.
\]
Moreover, by \eqref{bd3} and Rellich's compactness theorem, \( H^s_{\mathrm{per}}(\Omega)\Subset L^2_{\mathrm{per}}(\Omega)\), so for each fixed $t$, the set $\{u_N(t)\}_{N\in\mathbb N}$ is relatively
compact in $L^2_{\mathrm{per}}(\Omega)$. The Arzelà--Ascoli theorem in the metric space $L^2_{\mathrm{per}}(\Omega)$, equivalently Simon's compactness theorem \cite{Simon1987}, gives a subsequence, still denoted by $u_N$, and a function $\bar u$ such that
\[
    u_N\to \bar u \qquad\text{strongly in } C([0,\bar T];L^2_{\mathrm{per}}(\Omega)).
\]
In addition, the uniform bounds imply, after extracting a further subsequence
if necessary,
\[
    u_N\xrightharpoonup{\ast} \bar u
    \quad\text{in } L^\infty(0,\bar T;H^s_{\mathrm{per}}(\Omega)),
\]
and
\[
    \partial_tu_N\xrightharpoonup{\ast} \partial_t\bar u
    \quad\text{in } L^\infty(0,\bar T;L^2_{\mathrm{per}}(\Omega)).
\]
The weak-\(\ast\) convergences implies the regularity of the limit,  that is, \(\bar u\)
    in  \(L^\infty(0,\bar T;H^s_{\mathrm{per}}(\Omega))\) and \(\partial_t\bar u\)
    in  \(L^\infty(0,\bar T;L^2_{\mathrm{per}}(\Omega))\). 

We now pass to the limit in the Galerkin formulation. Let
$\varphi\in C_c^\infty([0,\bar T)\times\Omega)$ and choose
$v=\Pi_N\varphi\in \mathbb X_N$ in \eqref{fsg:fZK}. Integrating in time,
\[
    \int_0^{\bar T}
    \left[
    (\partial_tu_N,\Pi_N\varphi)
    +
    \left((-\Delta)^{\frac{\alpha}{2}}u_N,
    \partial_{x_1}\Pi_N\varphi\right)
    -
    \frac12
    \left(u_N^2,\partial_{x_1}\Pi_N\varphi\right)
    \right]dt=0.
\]
For the time derivative term, integration by parts gives
\[
    \int_0^{\bar T}(\partial_tu_N,\Pi_N\varphi)\,dt
    =
    -\int_0^{\bar T}(u_N,\partial_t\Pi_N\varphi)\,dt
    -(u_N(0),\Pi_N\varphi(0)).
\]
Since $u_N\to\bar u$ strongly in $C([0,\bar T];L^2)$,
$\Pi_N\varphi\to\varphi$ smoothly, and $u_N(0)=\Pi_Nu_0\to u_0$ in $L^2$,
we obtain
\[
    \int_0^{\bar T}(\partial_tu_N,\Pi_N\varphi)\,dt
    \to
    -\int_0^{\bar T}(\bar u,\partial_t\varphi)\,dt
    -(u_0,\varphi(0)).
\]
For the dispersive term, using the self-adjointness of
$(-\Delta)^{\alpha/2}$ and the commutation with $\partial_{x_1}$ and
$\Pi_N$,
\[
    \int_0^{\bar T}
    \left((-\Delta)^{\frac{\alpha}{2}}u_N,
    \partial_{x_1}\Pi_N\varphi\right)dt
    =
    \int_0^{\bar T}
    \left(u_N,
    (-\Delta)^{\frac{\alpha}{2}}\partial_{x_1}\Pi_N\varphi\right)dt .
\]
The strong $L^2$ convergence of $u_N$ and the smooth convergence of
$\Pi_N\varphi$ imply
\[
    \int_0^{\bar T}
    \left((-\Delta)^{\frac{\alpha}{2}}u_N,
    \partial_{x_1}\Pi_N\varphi\right)dt
    \to
    \int_0^{\bar T}
    \left(\bar u,
    (-\Delta)^{\frac{\alpha}{2}}\partial_{x_1}\varphi\right)dt .
\]
For the nonlinear term, the strong convergence in
$C([0,\bar T];L^2)$ and the uniform bound in
$L^\infty(0,\bar T;H^s_{\mathrm{per}}(\Omega))$ imply
\[
    u_N^2\to \bar u^2
    \qquad\text{strongly in }
    C([0,\bar T];L^2_{\mathrm{per}}(\Omega)).
\]
Indeed,
\[
    \|u_N^2-\bar u^2\|
    \leq
    \|u_N-\bar u\|
    \left(
    \|u_N\|_{L^\infty(\Omega)}
    +
    \|\bar u\|_{L^\infty(\Omega)}
    \right),
\]
and the factor in parentheses is uniformly bounded because $s>2$.
Therefore
\[
    \int_0^{\bar T}
    \left(u_N^2,\partial_{x_1}\Pi_N\varphi\right)dt
    \to
    \int_0^{\bar T}
    \left(\bar u^2,\partial_{x_1}\varphi\right)dt .
\]
Passing to the limit gives exactly \eqref{weaksoln}. Hence $\bar u$ is a
weak solution. Since \( \bar u\in L^\infty(0,\bar T;H^s_{\mathrm{per}}(\Omega))\) for \(s\geq 1+\alpha\), we have
\[
    (-\Delta)^{\frac{\alpha}{2}}\partial_{x_1}\bar u
    \in L^\infty(0,\bar T;L^2_{\mathrm{per}}(\Omega)).
\]
Since also $s>2$,
\[
    \bar u\,\partial_{x_1}\bar u
    \in L^\infty(0,\bar T;L^2_{\mathrm{per}}(\Omega)).
\]
Thus $\partial_t\bar u\in L^\infty(0,\bar T;L^2_{\mathrm{per}}(\Omega))$,
and the equation holds as an $L^2$-identity almost everywhere in time. Hence
$\bar u$ is a strong solution. It remains to prove uniqueness. Let $\bar u$ and $\bar w$ be two strong
solutions in \( L^\infty(0,\bar T;H^s_{\mathrm{per}}(\Omega))
    \cap W^{1,\infty}(0,\bar T;L^2_{\mathrm{per}}(\Omega))\) with $\bar{u}(0) = \bar{w}(0) = u_0$. Set $z=\bar u-\bar w$. Then
\[
    \partial_t z
    -
    (-\Delta)^{\frac{\alpha}{2}}\partial_{x_1}z
    +
    \bar u\,\partial_{x_1}\bar u
    -
    \bar w\,\partial_{x_1}\bar w
    =0,
    \qquad z(0)=0.
\]
Taking the $L^2$ inner product with $z$ and using the skew-symmetry of the
dispersive term,
\[
    \frac12\frac{d}{dt}\|z\|^2
    =
    -\left(
    \bar u\,\partial_{x_1}\bar u
    -
    \bar w\,\partial_{x_1}\bar w,
    z
    \right).
\]
Since \(\bar u\,\partial_{x_1}\bar u -\bar w\,\partial_{x_1}\bar w =
 z\,\partial_{x_1}\bar u + \bar w\,\partial_{x_1}z\), we have
\[
    \frac12\frac{d}{dt}\|z\|^2 =-(\partial_{x_1}\bar u,z^2) - (\bar w\,\partial_{x_1}z,z) = -(\partial_{x_1}\bar u,z^2) + \frac12(\partial_{x_1}\bar w,z^2).
\]
Therefore
\[
    \frac{d}{dt}\|z\|^2
    \leq
    C\left(
    \|\partial_{x_1}\bar u\|_{L^\infty(\Omega)}
    +
    \|\partial_{x_1}\bar w\|_{L^\infty(\Omega)}
    \right)\|z\|^2 .
\]
Since $s>2$, the embedding
$H^s_{\mathrm{per}}(\Omega)\hookrightarrow W^{1,\infty}(\Omega)$ gives
\[
    \frac{d}{dt}\|z\|^2
    \leq C\|z\|^2 .
\]
Gr\"onwall's inequality and $z(0)=0$ imply $z\equiv 0$ on $[0,\bar T]$.
Thus the limit is unique. Since every subsequence has a further subsequence converging to the same limit, the whole sequence $u_N$ converges to $\bar u$ in $C([0,\bar T];L^2_{\mathrm{per}}(\Omega))$. This completes the proof.
\end{proof}

\begin{remark}\label{remk_extnsn}
The convergence result above is local in the high Sobolev topology. This
local statement is valid for the full range \(0<\alpha\leq2\), provided
\(s\geq 2+\alpha\). The role of this assumption is to control the transport
nonlinearity through \(
H^s_{\mathrm{per}}(\Omega)\hookrightarrow W^{1,\infty}(\Omega)\), and to ensure that \( (-\Delta)^{\alpha/2}\partial_{x_1}u \in L^2_{\mathrm{per}}(\Omega)\).

At the semi-discrete level, each \(u_N\) exists for all positive times because
the Galerkin system is finite dimensional and the conserved \(L^2\)-norm
prevents blow-up of the Fourier coefficient vector for each fixed \(N\).
Moreover, for \(\alpha>1\), the conserved Hamiltonian gives a useful global
energy-level control. Indeed, the periodic Gagliardo--Nirenberg inequality gives
\[
    \|v\|_{L^3(\Omega)}^3
    \leq
    C\|v\|_{H^{\alpha/2}_{\mathrm{per}}(\Omega)}^{2/\alpha}
    \|v\|^{3-2/\alpha}.
\]
Since the \(L^2\)-norm controls the zero Fourier mode, \(
    \|v\|_{H^{\alpha/2}_{\mathrm{per}}(\Omega)} \leq C\left(\|(-\Delta)^{\alpha/4}v\|+\|v\|\right)\).
Using the conservation of \(\|u_N(t)\|\) in \eqref{cons2}, we therefore obtain
\[
    \left|\int_\Omega u_N^3\,d\mathbf{x}\right|
    \leq
    C\left(1+\|(-\Delta)^{\alpha/4}u_N(t)\|^{2/\alpha}\right).
\]
If \(\alpha>1\), then \(2/\alpha<2\), and Young's inequality yields
\[
    \left|\int_\Omega u_N^3\,d\mathbf{x}\right|
    \leq
    \varepsilon \|(-\Delta)^{\alpha/4}u_N(t)\|^2+C_\varepsilon.
\]
Consequently,
\[
\begin{aligned}
    \mathcal H(u_N(t))
    = \frac12\|(-\Delta)^{\alpha/4}u_N(t)\|^2-\frac16\int_\Omega u_N^3\,d\mathbf{x} \geq
    \left(\frac12-\frac{\varepsilon}{6}\right)\|(-\Delta)^{\alpha/4}u_N(t)\|^2
    -C_{\varepsilon}.
\end{aligned}
\]
Choosing \(\varepsilon<3\) and using the conservation of
\(\mathcal H(u_N(t))\), we obtain
\[
    \|(-\Delta)^{\alpha/4}u_N(t)\|^2
    \leq
    C\left(\mathcal H(u_N(0)),\|u_N(0)\|\right),
    \qquad t\geq0.
\]
Thus, the conservation laws \eqref{cons3} give a global-in-time bound for \( \|u_N(t)\|_{H^{\alpha/2}_{\mathrm{per}}(\Omega)}\) uniformly at the energy level. This is important for stability, but it does
not control the higher norm \(\|u_N(t)\|_s\) with \(s\geq2+\alpha\), which is
the norm required in the compactness and optimal error analysis.

This Hamiltonian coercivity argument does not close when \(0<\alpha\leq1\):
for \(\alpha=1\), the cubic term has the same order as the quadratic
Hamiltonian part, and for \(0<\alpha<1\), it is superquadratic relative to
the Hamiltonian norm.  

The standard continuation criterion remains the following: if a strong solution satisfies
\[
    \sup_{0\leq t<T^\ast}\|u(t)\|_s<\infty,
\]
then it can be continued beyond \(T^\ast\). Hence the convergence result
extends to any interval \([0,T]\) on which an independent \(H^s\)-bound for
the exact solution is available.
\end{remark}

\begin{proof}[Proof of Theorem \ref{thm:exisuq}]
The theorem follows directly from Lemma \ref{conv_lemma}. The lemma constructs a Galerkin limit in \(
L^\infty(0,\bar T;H^s_{\mathrm{per}}(\Omega)) \cap W^{1,\infty}(0,\bar T;L^2_{\mathrm{per}}(\Omega))\), identifies it as a strong solution of \eqref{eqn:fZK}, proves uniqueness, and then concludes convergence of the whole Galerkin sequence in $C([0,\bar T];L^2_{\mathrm{per}}(\Omega))$.
\end{proof}

\section{Optimal Error Analysis}\label{sec4}

In this section, we develop an optimal error estimate for the FSG method \eqref{fsg:fZK} applied to the fZK equation \eqref{eqn:fZK}. The key idea, inspired by Deng's work \cite{deng2017optimal} on the one-dimensional Kawahara equation, is to introduce a modified projection operator adapted to the transport-dispersive operator appearing in the error equation. In contrast to a standard $L^2$ projection, this modified projection cancels the linear residual and allows us to recover the optimal spectral rate. We give the full proof in the present two-dimensional fractional setting.

We shall use the following standard Fourier projection estimates. If
$z\in H^r_{\mathrm{per}}(\Omega)$ and $0\leq \ell\leq r$, then
\begin{equation}\label{eq:std_proj_est}
    \|z-\Pi_Nz\|_\ell
    \leq C N^{\ell-r}|z|_r .
\end{equation}
Moreover, for $v_N\in\mathbb X_N$ and $\ell\geq 0$, the inverse estimate
\begin{equation}\label{eq:inverse_general}
    \|v_N\|_\ell\leq C N^\ell\|v_N\|
\end{equation}
holds. If $z$ is analytic and periodic, then there exist positive constants $C$ and $\sigma$, independent of $N$, such that
\begin{equation}\label{eq:std_proj_exp}
    \|z-\Pi_Nz\|\leq C e^{-\sigma N}\|z\|,
\end{equation}
see \cite[Chap. 13]{hesthaven2017numerical}.
We introduce
\[
    \mathcal L z:=Kz+u\partial_{x_1}z-(-\Delta)^{\alpha/2}\partial_{x_1} z,
\]
 where $K > 0$ is a constant that will be chosen sufficiently large to ensure coercivity. Since $(-\Delta)^{\alpha/2}\partial_{x_1}$ is skew-symmetric on the periodic domain, the $L^2$ adjoint of $\mathcal L$ is
\[
    \mathcal L^\ast z
    =
    Kz-\partial_{x_1}(uz)+(-\Delta)^{\alpha/2}\partial_{x_1} z .
\]
Let $u \in C^1(\Omega)$. For any $w \in H^r_{\mathrm{per}}(\Omega)$ with
$r \geq 1+\alpha$, we define the modified projection
$\widetilde{\Pi}_N w \in \mathbb{X}_N$ as follows:
\begin{equation}\label{eq:proj_def}
    \left(
    K(\widetilde{\Pi}_N w - w)
    +
    u\partial_{x_1}(\widetilde{\Pi}_N w - w)
    -
    (-\Delta)^{\alpha/2}\partial_{x_1}(\widetilde{\Pi}_N w - w),
    v
    \right)=0,
    \quad \forall v \in \mathbb{X}_N .
\end{equation}
Equivalently,
\[
    (\mathcal L(\widetilde{\Pi}_Nw-w),v)=0,
    \qquad \forall v\in\mathbb X_N .
\]

For \(0\leq\beta<1\), we also use the one-dimensional fractional derivative
in the distinguished direction \(x_1\), defined by
\[
    |\partial_{x_1}|^\beta f(\mathbf{x})
    =
    \sum_{\mathbf{k}\in\mathbb Z^2}
    |k_1|^\beta \widehat f(\mathbf{k})e^{i\mathbf{k}\cdot\mathbf{x}}.
\]
We use the interpolation inequality, which implies 
\begin{equation}\label{eq: interp_est_beta}
    \||\partial_{x_1}|^\beta f\|
    \leq
    \|f\|^{1-\beta}\|\partial_{x_1}f\|^\beta,
    \qquad 0\leq\beta\leq1,
\end{equation}
and the product estimate implies
\[
    \||\partial_{x_1}|^\beta(fg)\|
    \leq
    C\left(
    \|f\|_{W^{1,\infty}(\Omega)}\|g\|
    +
    \|f\|_{L^\infty(\Omega)}
    \||\partial_{x_1}|^\beta g\|
    \right),
    \qquad 0\leq\beta<1.
\]

\begin{theorem}\label{thm:mod_proj_well_defined}
Let \(0<\alpha\leq2\), \(u\in W^{1,\infty}(\Omega)\), and
\(w\in H^r_{\mathrm{per}}(\Omega)\) with \(r\geq\alpha+1\). There exists
\[
    K_0:=1+\frac12\|\partial_{x_1}u\|_{L^\infty(\Omega)}
\]
such that, for every \(K\geq K_0\), the variational problem
\eqref{eq:proj_def} has a unique solution
\(\widetilde{\Pi}_N w\in\mathbb X_N\).
\end{theorem}

\begin{proof}
Let $\phi = \widetilde{\Pi}_N w \in \mathbb{X}_N$ be the unknown. Equation
\eqref{eq:proj_def} is equivalent to finding $\phi \in \mathbb{X}_N$ such that
for all $v \in \mathbb{X}_N$,
\[
    (K\phi + u\partial_{x_1}\phi - (-\Delta)^{\alpha/2}\partial_{x_1} \phi, v)
    =
    (Kw + u\partial_{x_1}w - (-\Delta)^{\alpha/2}\partial_{x_1} w, v).
\]
Define the bilinear form $B: \mathbb{X}_N \times \mathbb{X}_N \to \mathbb{R}$
by
\[
    B(\phi,\psi)
    =
    (K\phi + u\partial_{x_1}\phi - (-\Delta)^{\alpha/2}\partial_{x_1}\phi,\psi),
\]
and the bounded linear functional $F:\mathbb X_N\to\mathbb R$ by
\[
    F(\psi)
    =
    (Kw + u\partial_{x_1}w-(-\Delta)^{\alpha/2}\partial_{x_1} w,\psi).
\]
The functional $F$ is well-defined because $r\geq \alpha+1$. We first prove coercivity. For any $\phi\in\mathbb X_N$, periodicity gives
\[
    (u\partial_{x_1}\phi,\phi)
    =
    -\frac12(\partial_{x_1}u,\phi^2),
\]
and Lemma \ref{lem:frac_properties} gives \( ((-\Delta)^{\alpha/2}\partial_{x_1}\phi,\phi)=0\). Therefore,
\[
    B(\phi,\phi)
    =
    K\|\phi\|^2
    -
    \frac12(\partial_{x_1}u,\phi^2)
    \geq
    \left(
    K-\frac12\|\partial_{x_1}u\|_{L^\infty(\Omega)}
    \right)\|\phi\|^2 .
\]
For $K\geq K_0$, this yields
\[
    B(\phi,\phi)\geq \|\phi\|^2 .
\]
The bilinear form is bounded on the finite-dimensional space $\mathbb X_N$.
Indeed,
\[
\begin{aligned}
    |B(\phi,\psi)|
    &\leq
    K\|\phi\|\|\psi\|
    +
    \|u\|_{L^\infty(\Omega)}
    \|\partial_{x_1}\phi\|\|\psi\|
    +
    \|(-\Delta)^{\alpha/2}\partial_{x_1}\phi\|\|\psi\|  \\
    &\leq
    \left(
    K+C\|u\|_{L^\infty(\Omega)}N+CN^{\alpha+1}
    \right)\|\phi\|\|\psi\|,
\end{aligned}
\]
where we used the inverse inequalities for trigonometric polynomials.
Thus, $B$ is coercive and bounded on $\mathbb X_N$. Since $\mathbb X_N$ is a finite-dimensional Hilbert space, the Lax--Milgram theorem gives a unique $\phi\in\mathbb X_N$ satisfying $B(\phi,\psi)=F(\psi)$ for all
$\psi\in\mathbb X_N$. Hence, $\widetilde{\Pi}_Nw=\phi$ is well-defined.
\end{proof}

\begin{theorem}\label{thm:proj_error}
Let \(0<\alpha\leq2\), \(r\geq\alpha+1\), and \(0\leq\ell\leq r\).
Let \(u\in W^{1,\infty}(\Omega)\) and
\(w\in H^r_{\mathrm{per}}(\Omega)\). Assume that \(K\) is chosen as in
Theorem \ref{thm:mod_proj_well_defined}, and let
\(\widetilde{\Pi}_N\) be defined by \eqref{eq:proj_def}. Then there exists a
constant \(C>0\), independent of \(N\), such that
\begin{equation}\label{proj_erro}
    \|w-\widetilde{\Pi}_Nw\|_\ell
    \leq
    C N^{\ell-r}|w|_r .
\end{equation}
If, in addition, \(w\) is real-analytic on \(\Omega\) with an analytic
extension to a complex strip of width \(\sigma_0>0\), then there exist
positive constants \(C\) and \(\sigma\), independent of \(N\), such
that
\begin{equation}\label{proj_erro2}
    \|w-\widetilde{\Pi}_Nw\|
    \leq
    Ce^{-\sigma N}.
\end{equation}
\end{theorem}

\begin{proof}
Let \( \rho_N:=w-\Pi_Nw\) and \(\theta_N:=\widetilde{\Pi}_Nw-\Pi_Nw\).
Then \( w-\widetilde{\Pi}_Nw=\rho_N-\theta_N\).
By \eqref{eq:std_proj_est}, we have
\begin{equation}\label{eq:rho_bound}
    \|\rho_N\|_\ell
    \leq
    C N^{\ell-r}|w|_r,
    \qquad 0\leq \ell\leq r .
\end{equation}
It remains to estimate $\theta_N$. Since
\[
    (\mathcal L(\widetilde{\Pi}_Nw-w),v)=0,
    \qquad \forall v\in\mathbb X_N,
\]
we have
\begin{equation}\label{eq:theta_eq1}
    (\mathcal L\theta_N,v) = (\mathcal L\rho_N,v)  = (K\rho_N+u\partial_{x_1}\rho_N-(-\Delta)^{\alpha/2}\partial_{x_1}\rho_N,v)=
    (u\partial_{x_1}\rho_N,v),
    \qquad \forall v\in\mathbb X_N,
\end{equation}
where we have used the orthogonality \eqref{eq:proj_ortho} and the fact that \( (-\Delta)^{\alpha/2}\partial_{x_1} v\in\mathbb X_N\) for every \(v\in\mathbb X_N\).
We next prove an auxiliary adjoint estimate. Let $g_N\in\mathbb X_N$ and let
$\phi_N\in\mathbb X_N$ solve
\begin{equation}\label{eq:adjoint_problem}
    (\mathcal L^\ast\phi_N,v)=(g_N,v),
    \qquad \forall v\in\mathbb X_N .
\end{equation}
Let \( c_K:=K-\frac12\|\partial_{x_1}u\|_{L^\infty(\Omega)}\). Taking $v=\phi_N$ in \eqref{eq:adjoint_problem}, and using periodicity together with the skew-symmetry of $(-\Delta)^{\alpha/2}\partial_{x_1}$, we obtain
\[
\begin{aligned}
 (\mathcal L^\ast\phi_N,\phi_N) &= K\|\phi_N\|^2  - (\partial_{x_1}(u\phi_N),\phi_N) + ((-\Delta)^{\alpha/2}\partial_{x_1}\phi_N,\phi_N)  \\ &= K\|\phi_N\|^2 - \frac12(\partial_{x_1}u,|\phi_N|^2)  \\ &\geq c_K\|\phi_N\|^2 .
\end{aligned}
\]
Since $K\geq K_0$, we have $c_K\geq1$. Hence \( \|\phi_N\|^2 \leq |(g_N,\phi_N)| \leq \|g_N\|\|\phi_N\|\),
and therefore
\begin{equation}\label{eq:adjoint_L2}
    \|\phi_N\|\leq C\|g_N\|.
\end{equation}
We also need a bound for \(\partial_{x_1}\phi_N\). Since \(
    \mathcal L^\ast\phi_N = K\phi_N-\partial_{x_1}(u\phi_N) + (-\Delta)^{\alpha/2}\partial_{x_1}\phi_N\),
the discrete adjoint equation \eqref{eq:adjoint_problem} can be written in \(\mathbb X_N\) as
\[
    \left(K+(-\Delta)^{\alpha/2}\partial_{x_1}\right)\phi_N
    =
    g_N+\Pi_N\partial_{x_1}(u\phi_N).
\]
Applying \(\partial_{x_1} \left(K+(-\Delta)^{\alpha/2}\partial_{x_1}\right)^{-1}\) gives
\[
    \partial_{x_1}\phi_N
    =
    T_1g_N+T_2\Pi_N(u\phi_N),
\]
where \(T_1\) and \(T_2\) are Fourier multipliers with symbols
\[
    m_1(\mathbf{k})
    =
    \frac{i k_1}{K+i k_1|\mathbf{k}|^\alpha},
    \qquad
    m_2(\mathbf{k})
    =
    \frac{(ik_1)^2}{K+i k_1|\mathbf{k}|^\alpha}.
\]
The multiplier \(m_1\) is uniformly bounded for every \(0<\alpha\leq2\).
For \(m_2\), set
\[
    \beta:=(1-\alpha)_+ = \max\{1-\alpha,0\}.
\]
Then, for \(k_1\neq0\),
\[
    |m_2(\mathbf{k})|
    \leq
    \frac{|k_1|^2}{|k_1||\mathbf{k}|^\alpha}
    =
    |k_1|\,|\mathbf{k}|^{-\alpha}
    \leq
    C |k_1|^\beta .
\]
Thus
\[
    \|T_2 f\|
    \leq
    C\||\partial_{x_1}|^\beta f\|,
    \qquad 0\leq\beta<1.
\]
Consequently,
\[
\begin{aligned}
    \|\partial_{x_1}\phi_N\|
    &\leq
    C\|g_N\|
    +
    C\||\partial_{x_1}|^\beta(u\phi_N)\|  \\
    &\leq
    C\|g_N\|
    +
    C\|u\|_{W^{1,\infty}(\Omega)}\|\phi_N\|
    +
    C\|u\|_{L^\infty(\Omega)}
    \||\partial_{x_1}|^\beta\phi_N\|.
\end{aligned}
\]
By interpolation estimate \eqref{eq: interp_est_beta}, \(  \||\partial_{x_1}|^\beta\phi_N\| \leq \|\phi_N\|^{1-\beta}
    \|\partial_{x_1}\phi_N\|^\beta \).
Using the \(L^2\)-estimate \eqref{eq:adjoint_L2}, we obtain
\[
    \|\partial_{x_1}\phi_N\|
    \leq
    C\|g_N\|
    +
    C\|g_N\|^{1-\beta}
    \|\partial_{x_1}\phi_N\|^\beta .
\]
Since \(\beta<1\), Young's inequality gives
\[
    C\|g_N\|^{1-\beta}
    \|\partial_{x_1}\phi_N\|^\beta
    \leq
    \frac12\|\partial_{x_1}\phi_N\|+C\|g_N\|.
\]
Therefore, \( \|\partial_{x_1}\phi_N\| \leq C\|g_N\|\). Combining this with the \(L^2\)-estimate \eqref{eq:adjoint_L2} gives
\begin{equation}\label{eq:adjoint_estimate}
    \|\phi_N\|+\|\partial_{x_1}\phi_N\|
    \leq
    C\|g_N\|,
\end{equation}
where \(C\) is independent of \(N\).
We now take $g_N=\theta_N$ in \eqref{eq:adjoint_problem}. Then
\[
\|\theta_N\|^2 =(\theta_N,\theta_N)= (\mathcal L\theta_N,\phi_N) = (u\partial_{x_1}\rho_N,\phi_N)
    = -(\rho_N,\partial_{x_1}(u\phi_N)).
\]
Hence, by \eqref{eq:adjoint_estimate},
\[
    \|\theta_N\|^2\leq\|\rho_N\| \left(
    \|\partial_{x_1}u\|_{L^\infty(\Omega)}\|\phi_N\|
+\|u\|_{L^\infty(\Omega)}\|\partial_{x_1}\phi_N\|
    \right)    \leq C\|\rho_N\|\|\theta_N\|.
\]
If $\theta_N=0$ there is nothing to prove. Otherwise, dividing by
$\|\theta_N\|$ gives
\begin{equation}\label{eq:theta_L2}
    \|\theta_N\|
    \leq
    C\|\rho_N\|
    \leq
    C N^{-r}|w|_r .
\end{equation}
Using the inverse estimate \eqref{eq:inverse_general}, we obtain for
\(0\leq \ell\leq r\),
\begin{equation}\label{eq:theta_l}
    \|\theta_N\|_\ell
    \leq
    C N^\ell\|\theta_N\|
    \leq
    C N^{\ell-r}|w|_r .
\end{equation}
Finally,
\[
    \|w-\widetilde{\Pi}_Nw\|_\ell
    \leq
    \|\rho_N\|_\ell+\|\theta_N\|_\ell
    \leq
    C N^{\ell-r}|w|_r,
\]
which proves \eqref{proj_erro}. If $w$ is analytic, then \eqref{eq:std_proj_exp} gives \(
    \|\rho_N\|\leq C e^{-\sigma N}\).
The same duality argument leading to \eqref{eq:theta_L2} gives
\[
    \|\theta_N\|\leq C\|\rho_N\|
    \leq C e^{-\sigma N}.
\]
Therefore,
\[
    \|w-\widetilde{\Pi}_Nw\|
    \leq
    \|\rho_N\|+\|\theta_N\|
    \leq
    C e^{-\sigma N},
\]
which proves \eqref{proj_erro2}.
\end{proof}

\begin{lemma}\label{lem:time_error}
Let \(0<\alpha\leq2\), \(r\geq\alpha+1\), and \(0\leq\ell\leq r\).
Assume that
\[
    u\in C^1([0,T];W^{1,\infty}(\Omega)),
    \qquad
    w\in C^1([0,T];H^r_{\mathrm{per}}(\Omega)).
\]
Assume that \(K\) is chosen uniformly in time so that
\begin{equation}\label{eq:K_choice}
    K\geq
    1+\frac12\sup_{0\leq t\leq T}
    \|\partial_{x_1}u(t)\|_{L^\infty(\Omega)} .
\end{equation}
Then the modified projection error satisfies
\begin{equation}\label{time_error}
    \|\partial_t(w-\widetilde{\Pi}_Nw)(t)\|_\ell
    \leq
    C N^{\ell-r}
    \left(
    |\partial_t w(t)|_r+|w(t)|_r
    \right),
    \qquad 0\leq t\leq T,
\end{equation}
where \(C\) is independent of \(N\).
\end{lemma}

\begin{proof}
For each fixed $t$, the projection $\widetilde{\Pi}_N$ is defined by
\eqref{eq:proj_def} with coefficient $u(t)$. Since the problem is finite dimensional and its coefficients are $C^1$ in time, the map
$t\mapsto \widetilde{\Pi}_Nw(t)$ is differentiable.
Define
\[
    \chi_N(t):=\partial_t\widetilde{\Pi}_Nw(t)-\widetilde{\Pi}_N(\partial_t w )(t)\in\mathbb X_N .
\]
Then
\[
    \partial_t(w-\widetilde{\Pi}_Nw)
    =
    (\partial_t w-\widetilde{\Pi}_N(\partial_t w ))-\chi_N .
\]
 By Theorem \ref{thm:proj_error},
\begin{equation}\label{eq:pt_projection}
    \|\partial_t w(t)-\widetilde{\Pi}_N(\partial_t w )(t)\|_\ell
    \leq
    C N^{\ell-r}|\partial_t w(t)|_r .
\end{equation}
It remains to estimate $\chi_N$. Differentiating the projection identity
\[
    (\mathcal L(t)(\widetilde{\Pi}_Nw-w),v)=0,
    \qquad \forall v\in\mathbb X_N,
\]
with respect to time gives
\[
    \left(
    \mathcal L(t)
    \left(
    \partial_t\widetilde{\Pi}_Nw-\partial_t w
    \right),
    v
    \right)
    +
    \left(
    (\partial_tu)\partial_{x_1}
    (\widetilde{\Pi}_Nw-w),
    v
    \right)
    =0,
    \qquad \forall v\in\mathbb X_N .
\]
On the other hand, the definition of $\widetilde{\Pi}_N(\partial_t w )(t)$ gives
\[
    \left(
    \mathcal L(t)(\widetilde{\Pi}_N(\partial_t w )-\partial_t w),v
    \right)=0,
    \qquad \forall v\in\mathbb X_N .
\]
Subtracting the two identities, we obtain
\begin{equation}\label{eq:chi_eq}
    (\mathcal L(t)\chi_N,v)
    =
    -
    \left(
    (\partial_tu)\partial_{x_1}
    (\widetilde{\Pi}_Nw-w),
    v
    \right),
    \qquad \forall v\in\mathbb X_N .
\end{equation}
We now use the adjoint argument from the proof of Theorem \ref{thm:proj_error}. Let $\phi_N\in\mathbb X_N$ solve
\[
    (\mathcal L(t)^\ast\phi_N,v)=(\chi_N,v),
    \qquad \forall v\in\mathbb X_N .
\]
The estimate \eqref{eq:adjoint_estimate}, applied at time $t$, gives
\[
    \|\phi_N\|+\|\partial_{x_1}\phi_N\|
    \leq C\|\chi_N\|.
\]
Taking $v=\phi_N$ in \eqref{eq:chi_eq}, we obtain
\[
    \|\chi_N\|^2= (\mathcal L(t)\chi_N,\phi_N) =-\left(
    (\partial_tu)\partial_{x_1}
    (\widetilde{\Pi}_Nw-w), \phi_N
    \right)= \left(\widetilde{\Pi}_Nw-w, \partial_{x_1}\left((\partial_tu)\phi_N\right) \right).
\]
Therefore,
\[
    \|\chi_N\|^2 \leq \|\widetilde{\Pi}_Nw-w\|
    \left( \|\partial_{x_1}\partial_tu\|_{L^\infty(\Omega)}\|\phi_N\|
    + \|\partial_tu\|_{L^\infty(\Omega)}
    \|\partial_{x_1}\phi_N\|
    \right)\leq C\|\widetilde{\Pi}_Nw-w\|\|\chi_N\|.
\]
Using Theorem \ref{thm:proj_error} with $\ell=0$, we get
\[
    \|\chi_N\|
    \leq
    C\|\widetilde{\Pi}_Nw-w\|
    \leq
    C N^{-r}|w(t)|_r .
\]
By the inverse estimate \eqref{eq:inverse_general},
\[
    \|\chi_N\|_\ell
    \leq
    C N^\ell\|\chi_N\|
    \leq
    C N^{\ell-r}|w(t)|_r .
\]
Combining this estimate with \eqref{eq:pt_projection}, we obtain
\[
    \|\partial_t(w-\widetilde{\Pi}_Nw)(t)\|_\ell \leq
    \|\partial_t w(t)-\widetilde{\Pi}_N(\partial_t w )(t)\|_\ell+\|\chi_N(t)\|_\ell  \leq C N^{\ell-r}
    \left( |\partial_t w(t)|_r+|w(t)|_r\right).
\]
This proves \eqref{time_error}.
\end{proof}

\begin{lemma}\label{lem:time_error_analytic}
Assume that the hypotheses of Lemma \ref{lem:time_error} hold. In addition,
assume that $u$, $w$, and $\partial_t w$ are uniformly real-analytic in the
spatial variables on $[0,T]$. Then there exist positive constants $C$ and
$\sigma$, independent of $N$, such that
\begin{equation}\label{time_error_analytic}
    \|\partial_t(w-\widetilde{\Pi}_Nw)(t)\|
    \leq
    C e^{-\sigma N},
    \qquad 0\leq t\leq T .
\end{equation}
\end{lemma}

\begin{proof}
The proof follows the proof of Lemma \ref{lem:time_error}, replacing the
algebraic projection estimate \eqref{proj_erro} by the exponential estimate
\eqref{proj_erro2}. More precisely, differentiating \eqref{eq:proj_def} in
time gives the same equation as in Lemma \ref{lem:time_error}, with an
additional coefficient term involving
$\partial_tu\,\partial_{x_1}(w-\widetilde{\Pi}_Nw)$. The duality argument used
in Theorem \ref{thm:proj_error} and Lemma \ref{lem:time_error} gives
\[
    \|\partial_t(w-\widetilde{\Pi}_Nw)(t)\|
    \leq
    C\left(
    \|\partial_t w-\widetilde{\Pi}_N(\partial_t w)\|
    +
    \|w-\widetilde{\Pi}_Nw\|
    \right).
\]
Since $w$ and $\partial_t w$ are uniformly analytic in space, both terms on
the right-hand side are bounded by $Ce^{-\sigma N}$ by \eqref{proj_erro2}.
This proves \eqref{time_error_analytic}.
\end{proof}

\begin{theorem}[Spectral convergence]\label{thm:spectral_accuracy}
Let \(0<\alpha\leq2\), and let \(r>2+\alpha\). Assume that the exact solution
of \eqref{eqn:fZK} satisfies \(  u\in C^1([0,T];H^r_{\mathrm{per}}(\Omega))\).
Let \(u_N\) be the semi-discrete approximation obtained by \eqref{fsg:fZK}.
Assume that the constant \(K\) satisfies \eqref{eq:K_choice}. Then there exists a constant \( C=C\left(T,\alpha,r,K,\Omega,
    \|u\|_{C^1([0,T];H^r_{\mathrm{per}}(\Omega))}\right)\),
independent of \(N\), such that
\begin{equation}\label{eq:L2_error_bound}
    \|u(t)-u_N(t)\|
    \leq
    C N^{-r},
    \qquad 0\leq t\leq T .
\end{equation}
If, in addition, \(u\) and \(\partial_tu\) are uniformly real-analytic in the
spatial variables on \([0,T]\), then there exist positive constants \(C\) and \(\sigma\), independent of \(N\), such that
\[
    \|u(t)-u_N(t)\|
    \leq
    Ce^{-\sigma N},
    \qquad 0\leq t\leq T .
\]
\end{theorem}

\begin{proof}
For notational convenience, set
\[
    \eta_N:=u-\widetilde{\Pi}_Nu,
    \qquad
    \mathcal E_N:=u_N-\widetilde{\Pi}_Nu .
\]
Then \( u_N-u=\mathcal E_N-\eta_N \).
By Theorem \ref{thm:proj_error} and Lemma \ref{lem:time_error}, we have
\begin{align}
    \label{eta_l2}
    \|\eta_N(t)\|
    &\leq C N^{-r}|u(t)|_r,\\
    \label{eta_time}
    \|\partial_t\eta_N(t)\|
    &\leq C N^{-r}\left(|\partial_tu(t)|_r+|u(t)|_r\right),
\end{align}
for $0\leq t\leq T$. Choose a number \(\mu\) such that \(
    2<\mu<r\). This is possible because \(r>2+\alpha>2\). By Theorem \ref{thm:proj_error},
\[
    \|\eta_N(t)\|_\mu
    \leq
    C N^{\mu-r}|u(t)|_r .
\]
Since \(\mu>2\), the two-dimensional Sobolev embedding gives \(
    H^\mu_{\mathrm{per}}(\Omega) \hookrightarrow W^{1,\infty}(\Omega)\).
Therefore,
\begin{equation}\label{tilde_u_w1inf}
\begin{aligned}
    \|\partial_{x_1}\widetilde{\Pi}_Nu(t)\|_{L^\infty(\Omega)}
    &\leq
    \|\partial_{x_1}u(t)\|_{L^\infty(\Omega)}
    +
    \|\partial_{x_1}(u-\widetilde{\Pi}_Nu)(t)\|_{L^\infty(\Omega)} \\
    &\leq
    C\|u(t)\|_r
    +
    C\|\eta_N(t)\|_\mu \\
    &\leq
    C\|u(t)\|_r .
\end{aligned}
\end{equation}
The constant is independent of \(N\), since \(N^{\mu-r}\leq1\). We now derive the error equation. Subtracting the exact equation from the semi-discrete equation and using the defining relation \eqref{eq:proj_def} of $\widetilde{\Pi}_N$ with $w=u(t)$, we obtain, for all $v\in\mathbb X_N$,
\begin{equation}\label{error_eqed}
\begin{aligned}
    &\left(
    \partial_t\mathcal E_N
    -
    (-\Delta)^{\frac{\alpha}{2}}\partial_{x_1}\mathcal E_N
    +
    u_N\partial_{x_1}u_N
    -
    \widetilde{\Pi}_Nu\,\partial_{x_1}(\widetilde{\Pi}_Nu),
    v
    \right)  \\
    &\qquad
    +(K\eta_N,v)
    -
    \left(
    \eta_N\,\partial_{x_1}(\widetilde{\Pi}_Nu),
    v
    \right)
    -
    (\partial_t\eta_N,v)
    =0 .
\end{aligned}
\end{equation}
Taking $v=\mathcal E_N$ in \eqref{error_eqed} and using the
skew-symmetry property from Lemma \ref{lem:frac_properties}, we get
\begin{equation}\label{energy_errored}
\begin{aligned}
    \frac12\frac{d}{dt}\|\mathcal E_N\|^2 &= \left(
    \widetilde{\Pi}_Nu\,\partial_{x_1}(\widetilde{\Pi}_Nu)
    - u_N\partial_{x_1}u_N, \mathcal E_N \right) \\  &\quad  -(K\eta_N,\mathcal E_N) + \left(
    \eta_N\,\partial_{x_1}(\widetilde{\Pi}_Nu), \mathcal E_N
    \right)  + (\partial_t\eta_N,\mathcal E_N).
\end{aligned}
\end{equation}
We first estimate the nonlinear difference. Since \( u_N=\widetilde{\Pi}_Nu+\mathcal E_N\), we have
\begin{align*}
    \left(
    \widetilde{\Pi}_Nu\,\partial_{x_1}(\widetilde{\Pi}_Nu)
    -u_N\partial_{x_1}u_N,\mathcal E_N\right) 
    = -\left( \mathcal E_N\,\partial_{x_1}(\widetilde{\Pi}_Nu), \mathcal E_N\right)-\left(\widetilde{\Pi}_Nu\,\partial_{x_1}\mathcal E_N,
    \mathcal E_N \right)-\left(\mathcal E_N\,\partial_{x_1}\mathcal E_N, \mathcal E_N\right).
\end{align*}
By periodicity,
\[
    \left( \mathcal E_N\,\partial_{x_1}\mathcal E_N,
    \mathcal E_N\right)
    =
    \frac13\int_\Omega \partial_{x_1}(\mathcal E_N^3)\,d\mathbf{x}
    =0,
\]
and
\[
    -\left(\widetilde{\Pi}_Nu\,\partial_{x_1}\mathcal E_N,
    \mathcal E_N
    \right)= \frac12
    \left( \partial_{x_1}(\widetilde{\Pi}_Nu),
    \mathcal E_N^2
    \right).
\]
Therefore,
\[
    \left(
    \widetilde{\Pi}_Nu\,\partial_{x_1}(\widetilde{\Pi}_Nu)
    -u_N\partial_{x_1}u_N,\mathcal E_N\right)
    =
    -\frac12\left(\partial_{x_1}(\widetilde{\Pi}_Nu),
    \mathcal E_N^2\right).
\]
Using \eqref{tilde_u_w1inf}, we obtain
\begin{equation}\label{nonlinear_error_estimate}
    \left|
    \left(\widetilde{\Pi}_Nu\,\partial_{x_1}(\widetilde{\Pi}_Nu)
    -u_N\partial_{x_1}u_N,
    \mathcal E_N\right) \right| \leq C\|\mathcal E_N\|^2 .
\end{equation}
Next, using \eqref{eta_l2}, \eqref{eta_time}, and \eqref{tilde_u_w1inf}, we estimate
\begin{align}
    \label{K_eta_est}
    |(K\eta_N,\mathcal E_N)|
    &\leq
    C_K N^{-r}|u(t)|_r\|\mathcal E_N\|
    \leq
    C N^{-2r}+C\|\mathcal E_N\|^2,\\
    \label{eta_ux_est}
    \left|
    \left(
    \eta_N\,\partial_{x_1}(\widetilde{\Pi}_Nu),
    \mathcal E_N
    \right)
    \right|
    &\leq
    \|\eta_N\|
    \|\partial_{x_1}(\widetilde{\Pi}_Nu)\|_{L^\infty(\Omega)}
    \|\mathcal E_N\|                                      \notag\\
    &\leq
    C N^{-r}|u(t)|_r\|u(t)\|_r\|\mathcal E_N\|
    \leq
    C N^{-2r}+C\|\mathcal E_N\|^2,\\
    \label{eta_t_est}
    |(\partial_t\eta_N,\mathcal E_N)|
    &\leq
    C N^{-r}\left(|\partial_tu(t)|_r+|u(t)|_r\right)
    \|\mathcal E_N\|                                      \notag\\
    &\leq
    C N^{-2r}+C\|\mathcal E_N\|^2 .
\end{align}
Combining \eqref{energy_errored}--\eqref{eta_t_est}, we obtain
\begin{equation}\label{error_gronwall_ineq}
    \frac{d}{dt}\|\mathcal E_N(t)\|^2
    \leq
    C\|\mathcal E_N(t)\|^2
    +
    C N^{-2r},
    \qquad 0\leq t\leq T .
\end{equation}
It remains to estimate the initial error. Since $u_N(0)=\Pi_Nu_0$,
\[
    \mathcal E_N(0)
    =
    \Pi_Nu_0-\widetilde{\Pi}_Nu_0 .
\]
Thus, by the standard projection estimate and Theorem \ref{thm:proj_error},
\begin{equation}\label{initial_error_est}
\begin{aligned}
    \|\mathcal E_N(0)\| \leq
    \|\Pi_Nu_0-u_0\|+\|u_0-\widetilde{\Pi}_Nu_0\|\leq
    C N^{-r}|u_0|_r .
\end{aligned}
\end{equation}
Applying Gr\"onwall's's inequality to \eqref{error_gronwall_ineq} and using
\eqref{initial_error_est}, we obtain
\[
    \|\mathcal E_N(t)\|
    \leq
    C N^{-r},
    \qquad 0\leq t\leq T .
\]
Finally,
\[
    \|u(t)-u_N(t)\| \leq \|u(t)-\widetilde{\Pi}_Nu(t)\|
    + \|\widetilde{\Pi}_Nu(t)-u_N(t)\|
    = \|\eta_N(t)\|+\|\mathcal E_N(t)\|.
\]
Using \eqref{eta_l2}, we conclude
\[
    \|u(t)-u_N(t)\| \leq
    C N^{-r}, \qquad 0\leq t\leq T .
\]
This proves \eqref{eq:L2_error_bound}.

Now assume that $u$ and $\partial_tu$ are uniformly real-analytic in the
spatial variables on $[0,T]$. Then \eqref{proj_erro2} and
Lemma \ref{lem:time_error_analytic} give
\[
    \|\eta_N(t)\|
    +
    \|\partial_t\eta_N(t)\|
    \leq
    C e^{-\sigma N},
    \qquad 0\leq t\leq T .
\]
The same argument as above gives
\[
    \frac{d}{dt}\|\mathcal E_N(t)\|^2
    \leq
    C\|\mathcal E_N(t)\|^2
    +
    C e^{-2\sigma N}.
\]
Moreover,
\[
    \|\mathcal E_N(0)\|
    \leq
    \|\Pi_Nu_0-u_0\|
    +
    \|u_0-\widetilde{\Pi}_Nu_0\|
    \leq
    C e^{-\sigma N}.
\]
Gr\"onwall's's inequality therefore yields
\[
    \|\mathcal E_N(t)\|
    \leq
    C e^{-\sigma N},
    \qquad 0\leq t\leq T .
\]
Together with \eqref{proj_erro2}, this proves
\[
    \|u(t)-u_N(t)\|
    \leq
    C e^{-\sigma N},
    \qquad 0\leq t\leq T .
\]
The proof is complete.
\end{proof}

\section{Time Discretization by an Integrating-Factor RK4 Scheme}
\label{sec5}

The fZK equation \eqref{eqn:fZK} contains the stiff dispersive term
$\partial_{x_1}(-\Delta)^{\alpha/2}u$. After Fourier discretization, this
term produces purely imaginary eigenvalues of size $\mathcal O(N^{1+\alpha})$.
A standard explicit Runge--Kutta method applied directly to the semi-discrete
system would therefore require a severe time-step restriction of the form
$\Delta t\,N^{1+\alpha}\ll 1$. The integrating-factor approach removes this
stiffness by integrating the linear dispersive part exactly and applying the
explicit RK4 method only to the transformed nonlinear system. This is a standard
strategy for stiff and highly oscillatory spectral discretizations; see
\cite{canuto2006spectral,CoxMatthews2002,HochbruckOstermann2010,KassamTrefethen2005,vaissmoradi2009error}.

Let \(M\in\mathbb N\), \(\Delta t=T/M\), and \(t_n=n\Delta t\),
\(0\leq n\leq M\). The Fourier spectral Galerkin formulation \eqref{ode-sys}
yields a system of ordinary differential equations
\begin{equation}
    \frac{dU}{dt}(t)=G(U(t))=LU(t)+\mathcal N(U(t)),
    \label{eq:ode_system}
\end{equation}
where \(U(t)\) denotes the vector of Fourier coefficients of \(u_N(t)\). The
linear matrix \(L\) is diagonal with entries
\[
    L_{\mathbf m}=i m_1|\mathbf m|^\alpha,
    \qquad |\mathbf m|_\infty\leq N,
\]
and the nonlinear part \(\mathcal N\) is given by
\[
    \mathcal N_{\mathbf m}(U)
    = -\frac12 i m_1
    \sum_{\substack{|\mathbf k|_\infty\leq N\\
                    |\mathbf m-\mathbf k|_\infty\leq N}}
    \widehat u_N(\mathbf k)\widehat u_N(\mathbf m-\mathbf k),
    \qquad |\mathbf m|_\infty\leq N .
\]
Throughout this section, the norm of a coefficient vector is identified with
the \(L^2(\Omega)\)-norm of the corresponding trigonometric polynomial.

We introduce the integrating-factor variable
\[
    V(t)=e^{-Lt}U(t).
\]
Since \(L\) is constant and diagonal,
\begin{equation}\label{transsys_eq}
    \frac{dV}{dt}
    =
    e^{-Lt}\mathcal N(e^{Lt}V(t))
    =:\mathcal F(t,V(t)).
\end{equation}
The matrices \(e^{Lt}\) and \(e^{-Lt}\) are unitary in all periodic Sobolev
norms, because the entries \(L_{\mathbf m}\) are purely imaginary. Given \(V^n\approx V(t_n)\), one step of the classical fourth-order
Runge--Kutta method applied to \eqref{transsys_eq} is
\begin{equation}
    V^{n+1}=V^n+\frac16(k_1+2k_2+2k_3+k_4),
    \label{eq:rk4_update}
\end{equation}
where
\begin{align*}
    k_1&=\Delta t\,\mathcal F(t_n,V^n),\\
    k_2&=\Delta t\,\mathcal F\left(t_n+\frac{\Delta t}{2},
    V^n+\frac{k_1}{2}\right),\\
    k_3&=\Delta t\,\mathcal F\left(t_n+\frac{\Delta t}{2},
    V^n+\frac{k_2}{2}\right),\\
    k_4&=\Delta t\,\mathcal F(t_n+\Delta t,V^n+k_3).
\end{align*}
The numerical solution in the original Fourier variables is recovered by
\[
    U^{n+1}=e^{Lt_{n+1}}V^{n+1}.
\]
Thus, in terms of Fourier coefficients,
\begin{equation}\label{fdscheme}
    \widehat u_N^{\,n+1}(\mathbf m)
    =
    e^{L_{\mathbf m}t_{n+1}}\widehat v^{\,n+1}(\mathbf m),
    \qquad |\mathbf m|_\infty\leq N .
\end{equation}

\subsection*{Linearized stability}

We first record the exact linear stability restriction for the transformed
linearized problem. Replace the nonlinear term by
\(\lambda\partial_{x_1}u\), with \(\lambda\in\mathbb R\). In Fourier variables
the semi-discrete system becomes
\[
    \frac{d\widehat u(\mathbf k)}{dt}
    =  L_{\mathbf k}\widehat u(\mathbf k) - i\lambda k_1\widehat u(\mathbf k).
\]
After applying the integrating factor, each Fourier mode satisfies
\[
    \frac{d\widehat v(\mathbf k)}{dt} = -i\lambda k_1\widehat v(\mathbf k).
\]
The RK4 method applied to this equation gives
\[
    \widehat v^{\,n+1}(\mathbf k)
    = R(-i\lambda k_1\Delta t)\widehat v^{\,n}(\mathbf k),
\]
where \(R(z)\) is the stability function of RK4:
\[
    R(z)=1+z+\frac{z^2}{2}+\frac{z^3}{6}+\frac{z^4}{24}.
\]
For \(z=iy\), \(y\in\mathbb R\),
\[
    R(iy)
    =
    \left(1-\frac{y^2}{2}+\frac{y^4}{24}\right)
    +
    i\left(y-\frac{y^3}{6}\right),
\]
and hence
\[
    |R(iy)|^2
    =
    \left(1-\frac{y^2}{2}+\frac{y^4}{24}\right)^2
    +
    \left(y-\frac{y^3}{6}\right)^2
    =
    1+\frac{y^6(y^2-8)}{576}.
\]
Therefore \(|R(iy)|\leq 1\) for \(|y|\leq 2\sqrt2\). Since
\(|k_1|\leq N\), the linearized stability condition is
\begin{equation}\label{timestep}
    \Delta t
    \leq
    \frac{2\sqrt2}{|\lambda|N}.
\end{equation}
This is a significantly weaker restriction compared with the original dispersive stiffness, and it is also independent of the fractional order \(\alpha\).

\begin{theorem}\label{thm:linear_IFRK4_stability}
For the linearized fZK equation with \(\lambda\in\mathbb R\), the integrating-factor RK4 scheme is linearly stable under the condition
\eqref{timestep}.
\end{theorem}

\begin{proof}
For each Fourier mode, the amplification factor is
\(R(-i\lambda k_1\Delta t)\). The identity
\[
    |R(iy)|^2=1+\frac{y^6(y^2-8)}{576}
\]
shows that \(|R(iy)|\leq1\) whenever \(|y|\leq2\sqrt2\). Since
\(|k_1|\leq N\), condition \eqref{timestep} guarantees stability for every
Fourier mode in \(\mathbb X_N\).
\end{proof}

\begin{remark}\label{rem:linear_stability_not_enough}
The estimate \eqref{timestep} is a linearized stability condition. For the
nonlinear equation, the effective transport speed depends on the computed
solution. Therefore, in practical computations, one uses a convective
restriction of the form
\[
    \Delta t\,N
    \max_{0\leq n\leq M}\|u_N^n\|_{L^\infty(\Omega)}
    \leq c_0,
\]
with a moderate constant \(c_0>0\). The fully nonlinear error estimate below does not rely only on the linear stability condition. It also uses the consistency of one RK4 step and a nonlinear stability estimate for the semi-discrete flow.
\end{remark}

Let \(S_N(t;t_*,Z)\) denote the exact solution at time \(t\) of the transformed system \eqref{transsys_eq}, starting from \(Z\in\mathbb X_N\) at time \(t_*\).
Thus
\[
    S_N(t_*;t_*,Z)=Z,
    \qquad
    \frac{d}{dt}S_N(t;t_*,Z)=\mathcal F(t,S_N(t;t_*,Z)).
\]
Let \(\Psi_N(t_*,h,Z)\) denote one RK4 step of length \(h\) applied to \eqref{transsys_eq}, starting from \(Z\) at time \(t_*\). Hence the fully discrete method can be written as
\[
    V^{n+1}=\Psi_N(t_n,\Delta t,V^n).
\]
We also define the RK4 stage values for a variable step length \(h\). Let \( Z_1(h)=Z,\) and define the RK4 increments
\begin{equation}\label{eq:Z1_-Z4}
    \begin{aligned}
    K_1(h)&=h\mathcal F(t_*, Z_1(h)),\\
     Z_2(h)&=Z+\frac12K_1(h),
    \qquad
    K_2(h)=h\mathcal F\left(t_*+\frac h2, Z_2(h)\right),\\
     Z_3(h)&=Z+\frac12K_2(h),
    \qquad
    K_3(h)=h\mathcal F\left(t_*+\frac h2, Z_3(h)\right),\\
     Z_4(h)&=Z+K_3(h),
    \qquad
    K_4(h)=h\mathcal F(t_*+h, Z_4(h)).
 \end{aligned}
\end{equation}
Thus
\[
    \Psi_N(t_*,h,Z) = Z+\frac16\Big[K_1(h)+2K_2(h)+2K_3(h)+K_4(h)\Big].
\]

The local truncation error is obtained by starting one RK4 step from the exact semi-discrete value. Therefore, if \(V_N(t)=e^{-Lt}U_N(t)\), we define
\begin{equation}\label{eq:local_trunc}
    \tau_{n+1}
    :=
    S_N(t_{n+1};t_n,V_N(t_n))
    -
    \Psi_N(t_n,\Delta t,V_N(t_n)).
\end{equation}

\subsection*{Local truncation error}

The next result proves the fourth-order local consistency of the integrating-factor RK4 method \eqref{fdscheme}, with a constant that is independent of $N$.

\begin{remark}[Regularity required for the local temporal error]
\label{rem:regularity_count_ifrk4}
The Sobolev regularity required in the following theorem is only used for the fully discrete temporal error estimate. It is not needed for the semi-discrete compactness argument or for the optimal spatial projection estimate.

The Fourier multiplier \(L\) has a differential order \(1+\alpha\), while the nonlinear term
\[
    \mathcal N(Y)=-\frac12\Pi_N\partial_{x_1}(Y^2)
\]
loses one spatial derivative. In the transformed equation
\[
    V_t=\mathcal F(t,V)=e^{-Lt}\mathcal N(e^{Lt}V),
\]
each explicit time differentiation of \(\mathcal F\) introduces one factor of \(L\), and hence costs \(1+\alpha\) spatial derivatives. Since the fourth-order Runge--Kutta local error is of size \(\mathcal O(h^5)\), the relevant elementary differentials involve derivatives of the vector field up to order four. Thus, the highest spatial order that must be controlled is \(4(1+\alpha)+1\), where the additional \(1\) comes from the derivative \(\partial_{x_1}(Y^2)\). In two dimensions, we also need more than one additional derivative for the embedding
\[
    H^\sigma_{\mathrm{per}}(\Omega)\hookrightarrow L^\infty(\Omega),
    \qquad \sigma>1.
\]
Therefore the RK4 stage polynomials must be controlled in
\(H^{q_0}_{\mathrm{per}}(\Omega)\) with
\[
    q_0>4(1+\alpha)+2.
\]
Since the stage estimates are derived from the starting value and the nonlinear map \(\mathcal N\) loses one derivative at each stage, we take three extra derivatives and assume
\[
    q=q_0+3>4(1+\alpha)+5.
\]
This condition is a convenient sufficient condition, not an optimal sharp regularity threshold. Equivalently, one may replace it by the direct assumption that the transformed semi-discrete solution has uniformly bounded time derivatives up to order five and that the RK4 stage values remain in a bounded high Sobolev ball. For analytic solutions, these requirements are automatically
satisfied.
\end{remark}

\begin{theorem}\label{thm:local_truncation}
Let \(0<\alpha\leq2\), and let \(q>4(1+\alpha)+5\). Fix \(t_*\in[0,T]\), \(Z\in\mathbb X_N\),
\(0<h\leq T-t_*\), and let \(Y_N(t)\) be such that
\[
    Y_N(t)=e^{Lt}S_N(t;t_*,Z),
\]
and assume that \(  \sup_{t_*\leq t\leq t_*+h}\|Y_N(t)\|_q\leq R\) for some \(R>0\) independent of \(N\). Then there exists
\(h_0=h_0(R,q,\alpha,\Omega)>0\) such that, for \(0<h\leq h_0\),
\begin{equation}\label{eq:local_error_general}
    \left\|
    S_N(t_*+h;t_*,Z)-\Psi_N(t_*,h,Z)
    \right\|
    \leq
    C_R h^5,
\end{equation}
where \(C_R>0\) is independent of \(N\) and \(h\). In particular, taking
\(t_*=t_n\), \(h=\Delta t\), and \(Z=V_N(t_n)\), we obtain
\begin{equation}\label{eq:local_error_order5}
    \|\tau_{n+1}\|
    \leq
    C_R\Delta t^5 .
\end{equation}
\end{theorem}

\begin{proof}
We first record two elementary estimates for the nonlinear term. If
\(Y\in H^{\sigma+1}_{\mathrm{per}}(\Omega)\) and \(\sigma>1\), then
\[
    \|\mathcal N(Y)\|_\sigma
    =
    \left\|
    \frac12\Pi_N\partial_{x_1}(Y^2)
    \right\|_\sigma
    \leq
    C\|Y^2\|_{\sigma+1}
    \leq
    C\|Y\|_{\sigma+1}^2,
\]
where we used the boundedness of \(\Pi_N\) and the Sobolev embedding in two dimensions. The constant \(C\) is independent of \(N\). Let \(q_0:=q-3\). Then \(q_0>4(1+\alpha)+2\). We show that the RK4 internal stage polynomials are uniformly bounded in \(H^{q_0}_{\mathrm{per}}(\Omega)\). Let \( Z_1, Z_2, Z_3, Z_4\) be the transformed RK4 stage values defined above in \eqref{eq:Z1_-Z4}, and set
\[
    Y_j=e^{L(t_*+c_jh)}Z_j, \quad\text{ for }\quad j=1,2,3,4, \qquad c_1=0,\quad c_2=c_3=\frac12,\quad c_4=1.
\]
Thus \(Y_1=Y_N(t_*)\). Using the definitions of the stages and the unitarity of
\(e^{Lt}\), we have
\[
    \|Y_2\|_\sigma
    \leq
    \|Y_1\|_\sigma
    +
    \frac h2\|\mathcal N(Y_1)\|_\sigma
    \leq
    \|Y_1\|_\sigma
    +
    Ch\|Y_1\|_{\sigma+1}^2,
\]
for \(0\leq\sigma\leq q_0+2\). Similarly,
\[
    \|Y_3\|_\sigma
    \leq
    \|Y_1\|_\sigma
    +
    Ch\|Y_2\|_{\sigma+1}^2,
    \qquad
    \|Y_4\|_\sigma
    \leq
    \|Y_1\|_\sigma
    +
    Ch\|Y_3\|_{\sigma+1}^2 .
\]
Since \(q=q_0+3\) and \(\|Y_1\|_q\leq R\), these estimates, applied successively with \(\sigma=q_0+2,q_0+1,q_0\), imply that
\[
    \max_{1\leq j\leq4}\|Y_j\|_{q_0}\leq C_R
\]
for \(0<h\leq h_0\), with \(h_0\) depending only on
\(R,q,\alpha,\Omega\), and not on \(N\).

Next we prove the uniform bounds needed for the Taylor remainder. Since
\[
    \mathcal F(t,Z)=e^{-Lt}\mathcal N(e^{Lt}Z),
\]
every explicit \(t\)-derivative of \(\mathcal F\) is obtained by differentiating
the factors \(e^{-Lt}\), \(e^{Lt}\), and the quadratic map
\(\mathcal N\). Thus, for \(0\leq a\leq4\) and \(0\leq b\leq2\), each term in
\(\partial_t^aD_Z^b\mathcal F(t,Z)\) is a finite linear combination of
expressions of the form
\begin{equation}\label{eq:auxT_frechet}
T= e^{-Lt}L^{\nu_0}\Pi_N\partial_{x_1} \left[ (L^{\nu_1}Y^{(1)})(L^{\nu_2}Y^{(2)}) \right],
\qquad \nu_0+\nu_1+\nu_2\le a\le4,
\end{equation}
where \(Y^{(1)}\) and \(Y^{(2)}\) are physical polynomials associated with
the arguments of the Fr\'echet derivative, and \(D_Z^b\mathcal F(t,Z)\) denotes the \(b\)-th Fr\'echet derivative of \(\mathcal F\) with respect to its second variable \(Z\). Since \(L\) is a Fourier multiplier of differential order \(1+\alpha\), the total number of spatial derivatives falling on the product is at most
\[
    a(1+\alpha)+1 \leq 4(1+\alpha)+1 < q_0-1 .
\]
Since \(e^{-Lt}\) is unitary in all Sobolev norms and \(\Pi_N\) is
uniformly bounded, from \eqref{eq:auxT_frechet} we have
\[
\|T\| \le C \left\| \Lambda^{\nu_0(1+\alpha)+1} \left[ (L^{\nu_1}Y^{(1)})(L^{\nu_2}Y^{(2)})\right] \right\|.
\]
Applying the product estimate \eqref{eq:KP_product} with
\(s=\nu_0(1+\alpha)+1\), we obtain
\[
\|T\| \le C
\|L^{\nu_1}Y^{(1)}\|_{L^\infty}
\|L^{\nu_2}Y^{(2)}\|_{\nu_0(1+\alpha)+1} + C \|L^{\nu_2}Y^{(2)}\|_{L^\infty}
\|L^{\nu_1}Y^{(1)}\|_{\nu_0(1+\alpha)+1}.
\]
Since \((\nu_0+\nu_i)(1+\alpha)+1\le a(1+\alpha)+1\le4(1+\alpha)+1<q_0-1\), we have
\[
\|L^{\nu_i}Y\|_{\nu_0(1+\alpha)+1}\le C\|Y\|_{q_0}.
\]
Also, because \(q_0-\nu_i (1+\alpha)\ge q_0-4(1+\alpha)>2\), the embedding
\(H^{q_0-\nu_i (1+\alpha)}_{\mathrm{per}}(\Omega)\hookrightarrow L^\infty(\Omega)\)
gives
\[
\|L^{\nu_i}Y\|_{L^\infty}\le C\|Y\|_{q_0}.
\]
Consequently,
\[
\|T\|
\le
C
\|Y^{(1)}\|_{q_0}
\|Y^{(2)}\|_{q_0},
\]
with \(C\) independent of \(N\).
For \(b=0\), both factors \(Y^{(1)}\) and \(Y^{(2)}\) are the physical
base polynomial \(Y=e^{Lt}Z\), and the assumed \(H^{q_0}\)-bound gives
\[
\|\partial_t^a\mathcal F(t,Z)\|\le C_R.
\]
For \(b=1,2\), let \(W_1,\ldots,W_b\in \mathbb X_N\) be arbitrary
direction vectors in the transformed variable. For \(b=1\), one factor is the base polynomial \(Y=e^{Lt}Z\), while the
other is \(e^{Lt}W_1\). Thus
\[
\|\partial_t^aD_Z\mathcal F(t,Z)W_1\|
\le
C_R\|e^{Lt}W_1\|_{q_0}.
\]
For \(b=2\), the two factors are \(e^{Lt}W_1\) and \(e^{Lt}W_2\), and
therefore
\[
\|\partial_t^aD_Z^2\mathcal F(t,Z)[W_1,W_2]\|
\le
C
\|e^{Lt}W_1\|_{q_0}
\|e^{Lt}W_2\|_{q_0}.
\]
Combining these three cases gives
\[
    \left\|
    \partial_t^aD_Z^b\mathcal F(t,Z)
    [W_1,\ldots,W_b]
    \right\|
    \leq
    C_R\prod_{j=1}^b\|e^{Lt}W_j\|_{q_0},
\]
for \(0\le a\le4\) and \(0\le b\le2\), with the usual convention that
the empty product is \(1\). Since \(\mathcal N\) is quadratic,
\(D_Z^b\mathcal F=0\) for \(b\ge3\). 
 In particular, the exact transformed solution \(S_N(t;t_*,Z)\) and the RK4 stage expressions have all derivatives required below bounded in \(L^2(\Omega)\), uniformly in \(N\).

Define the one-step error as a function of the step length:
\[
    \mathcal R_N(h) := S_N(t_*+h;t_*,Z)-\Psi_N(t_*,h,Z).
\]
The classical fourth-order Runge--Kutta method satisfies all order conditions up to order four for non-autonomous systems. Equivalently, one applies RK4 to the autonomous augmented system
\[
    \frac{d}{dt}
    \begin{pmatrix}
        \theta\\ V
    \end{pmatrix}
    =
    \begin{pmatrix}
        1\\ \mathcal F(\theta,V)
    \end{pmatrix}.
\]
Therefore
\[
    \mathcal R_N(0)=0,
    \qquad
    \frac{d^j}{dh^j}\mathcal R_N(0)=0,
    \qquad j=1,2,3,4.
\]
The fifth derivative of the one-step error involves the fifth time derivative of the exact transformed solution and the derivatives of the RK4 stage expressions. Since RK4 has order four, the corresponding elementary differentials contain derivatives of the vector field only up to order four. The preceding bounds therefore control all terms appearing in $d^5\mathcal R_N/d\xi^5$. The derivative bounds established above imply
\[
    \sup_{0\leq \xi\leq h}
    \left\|
    \frac{d^5}{d\xi^5}\mathcal R_N(\xi)
    \right\|
    \leq C_R .
\]
Taylor's formula with integral remainder gives
\[
    \mathcal R_N(h)
    =
    \frac{h^5}{4!}
    \int_0^1
    (1-\theta)^4
    \frac{d^5}{d\xi^5}
    \mathcal R_N(\theta h)\,d\theta .
\]
Hence
\[
    \|\mathcal R_N(h)\|\leq C_Rh^5.
\]
This proves \eqref{eq:local_error_general}, and
\eqref{eq:local_error_order5} follows by choosing
\(t_*=t_n\), \(h=\Delta t\), and \(Z=V_N(t_n)\).
\end{proof}

The local estimate above controls only a single time step. To accumulate the local errors over all time steps, we need a stability estimate for the exact semi-discrete flow. The following lemma gives this estimate in \(L^2\), using the skew-symmetry of the fractional dispersive operator.

\begin{lemma}\label{lem:semidiscrete_flow_stability}
Let \(p_N,q_N\in C([t_*,t_*+h];\mathbb X_N)\) be two exact solutions of the
semi-discrete Galerkin equation \eqref{fsg:fZK}. Assume that
\[
    \sup_{t_*\leq t\leq t_*+h}
    \left(
    \|\partial_{x_1}p_N(t)\|_{L^\infty(\Omega)}
    +
    \|\partial_{x_1}q_N(t)\|_{L^\infty(\Omega)}
    \right)
    \leq R .
\]
Then
\begin{equation}\label{eq:semidiscrete_flow_stability}
    \|p_N(t)-q_N(t)\|
    \leq
    e^{CR(t-t_*)}
    \|p_N(t_*)-q_N(t_*)\|,
    \qquad t_*\leq t\leq t_*+h,
\end{equation}
where \(C>0\) is independent of \(N\).
\end{lemma}

\begin{proof}
Set \(z_N=p_N-q_N\). Subtracting the two semi-discrete equations gives
\[
    \partial_tz_N
    -
    (-\Delta)^{\alpha/2}\partial_{x_1}z_N
    +
    p_N\partial_{x_1}p_N
    -
    q_N\partial_{x_1}q_N
    =0
\]
in the Galerkin sense. Taking the \(L^2\)-inner product with \(z_N\), the
fractional dispersive term vanishes by Lemma \ref{lem:frac_properties}. Since
\[ p_N\partial_{x_1}p_N-q_N\partial_{x_1}q_N=z_N\partial_{x_1}p_N+q_N\partial_{x_1}z_N,
\]
we obtain
\[
    \frac12\frac{d}{dt}\|z_N\|^2=
    -\left(z_N\partial_{x_1}p_N,z_N\right)
    -\left(q_N\partial_{x_1}z_N,z_N\right) =
    -\left(\partial_{x_1}p_N,z_N^2\right)
    +\frac12\left(\partial_{x_1}q_N,z_N^2\right).
\]
Thus
\[
\frac{d}{dt}\|z_N\|^2 \leq   C \left(
\|\partial_{x_1}p_N\|_{L^\infty(\Omega)}+\|\partial_{x_1}q_N\|_{L^\infty(\Omega)}\right)\|z_N\|^2
\leq CR\|z_N\|^2.
\]
Gr\"onwall's inequality gives \eqref{eq:semidiscrete_flow_stability}.
\end{proof}

\subsection*{Temporal error of the nonlinear IFRK4 scheme}

We now combine the local truncation error with the semi-discrete flow stability.
The assumption below is the standard nonlinear stability requirement for an
explicit time integrator applied to a derivative nonlinearity: the exact
semi-discrete solution and the local semi-discrete flows issued from the
computed values must remain in a fixed high Sobolev ball on the time interval
under consideration.

\begin{theorem}\label{thm:temporal_error_ifrk4}
Let \(0<\alpha\leq2\), and let \(q>4(1+\alpha)+5\). Let \(u_N(t)\) be the exact semi-discrete solution of \eqref{fsg:fZK}, and let \(u_N^n\) be the
fully discrete solution generated by \eqref{eq:rk4_update}--\eqref{fdscheme},
with \(u_N^0=u_N(0)=\Pi_Nu_0\).
Assume that there exists a constant \(R>0\), independent of \(N\) and
\(\Delta t\), such that
\[
    \sup_{0\leq t\leq T}\|u_N(t)\|_q\leq R.
\]
Moreover, for each \(n\), let \(w_N^n(t)\) be the exact semi-discrete solution on \([t_n,t_{n+1}]\) with initial value \( w_N^n(t_n)=u_N^n\). Assume that
\[
    \sup_{0\leq n\leq M-1} \sup_{t_n\leq t\leq t_{n+1}} \|w_N^n(t)\|_q \leq R.
\]
Then, for \(\Delta t\leq h_0(R,q,\alpha,\Omega)\), there exists a constant \(C>0\), independent of \(N\) and \(\Delta t\), such that
\begin{equation}\label{eq:temporal_error_order4}
    \max_{0\leq n\leq M}
    \|u_N(t_n)-u_N^n\|
    \leq
    C\Delta t^4 .
\end{equation}
\end{theorem}

\begin{proof}
Let \( V_N(t)=e^{-Lt}U_N(t)\) and \(V^n=e^{-Lt_n}U^n\).
Since \(e^{Lt}\) is unitary in \(L^2(\Omega)\),
\[
    \|V_N(t_n)-V^n\|
    =
    \|u_N(t_n)-u_N^n\|.
\]
Define \( E^n:=V_N(t_n)-V^n \). The exact transformed solution satisfies
\[
    V_N(t_{n+1})
    =
    S_N(t_{n+1};t_n,V_N(t_n)),
\]
whereas the numerical method gives
\[
    V^{n+1}
    =
    \Psi_N(t_n,\Delta t,V^n).
\]
Therefore
\[
\begin{aligned}
    E^{n+1}
    &=S_N(t_{n+1};t_n,V_N(t_n))-\Psi_N(t_n,\Delta t,V^n) \\&=
    \left[ S_N(t_{n+1};t_n,V_N(t_n))-S_N(t_{n+1};t_n,V^n)
    \right]+
    \left[S_N(t_{n+1};t_n,V^n) -\Psi_N(t_n,\Delta t,V^n)
    \right].
\end{aligned}
\]
The first bracket is controlled by Lemma \ref{lem:semidiscrete_flow_stability}.
Indeed, the corresponding physical solutions are the exact semi-discrete flows
starting from \(u_N(t_n)\) and \(u_N^n\), respectively. Since \(q>2\), the Sobolev embedding \(H^q_{\mathrm{per}}(\Omega)\hookrightarrow W^{1,\infty}(\Omega)\)
and the assumed \(H^q\)-bounds imply
\[
    \left\| S_N(t_{n+1};t_n,V_N(t_n))- S_N(t_{n+1};t_n,V^n)
    \right\| \leq e^{C\Delta t}\|E^n\|.
\]
The second bracket is the one-step local defect starting from the numerical value \(V^n\). Since the exact semi-discrete flow \(w_N^n(t)\) issued from \(u_N^n\) is bounded by \(R\) in \(H^q\), Theorem \ref{thm:local_truncation}
gives
\[
    \left\|
    S_N(t_{n+1};t_n,V^n)
    -
    \Psi_N(t_n,\Delta t,V^n)
    \right\|
    \leq
    C\Delta t^5 .
\]
Consequently,
\[
    \|E^{n+1}\|
    \leq
    e^{C\Delta t}\|E^n\|
    +
    C\Delta t^5 .
\]
Since \(E^0=0\), the discrete Gr\"onwall's inequality gives
\[
    \|E^n\|
    \leq
    C\Delta t^4,
    \qquad 0\leq n\leq M.
\]
Using again the unitarity of \(e^{Lt_n}\), we obtain
\[
    \|u_N(t_n)-u_N^n\|
    =
    \|V_N(t_n)-V^n\|
    \leq
    C\Delta t^4 .
\]
This proves \eqref{eq:temporal_error_order4}.
\end{proof}

\begin{remark}\label{rem:temporal_assumptions}
The theorem separates the two roles of the time step. The integrating factor
removes the stiff dispersive restriction \(\Delta t\,N^{1+\alpha}\ll1\).
The remaining restriction is nonlinear and convective in nature; in practice
one chooses
\[
    \Delta t\,N\max_n\|u_N^n\|_{L^\infty(\Omega)}
    \leq c_0
\]
so that the numerical values and the RK4 stage values remain in the same
regularity regime. This is consistent with the linearized estimate
\eqref{timestep}, but the nonlinear convergence proof requires the high-norm
boundedness stated in Theorem \ref{thm:temporal_error_ifrk4}.
\end{remark}

\subsection*{Fully discrete spatial--temporal error}

We now combine the optimal spatial estimate from Section \ref{sec4} with the
fourth-order temporal estimate above.

\begin{theorem}\label{thm:fully_discrete_error}
Let \(0<\alpha\leq2\), \(r\geq\alpha+2\), and assume that the exact solution
\(u\) satisfies the hypotheses of Theorem \ref{thm:spectral_accuracy} on
\([0,T]\). Let \(u_N^n\) be generated by the integrating-factor RK4 scheme
\eqref{eq:rk4_update}--\eqref{fdscheme}. Assume, in addition, that the
hypotheses of Theorem \ref{thm:temporal_error_ifrk4} hold for the
semi-discrete and fully discrete trajectories. Then
\begin{equation}\label{eq:fully_discrete_error}
    \|u(t_n)-u_N^n\|
    \leq
    C\left(N^{-r}+\Delta t^4\right),
    \qquad 0\leq n\leq M,
\end{equation}
where \(C\) is independent of \(N\) and \(\Delta t\).

If \(u\) and \(\partial_tu\) are uniformly real-analytic in the spatial
variables on \([0,T]\), then there exist positive constants \(C\) and
\(\sigma\), independent of \(N\) and \(\Delta t\), such that
\begin{equation}\label{eq:fully_discrete_exp}
    \|u(t_n)-u_N^n\|
    \leq
    C\left(e^{-\sigma N}+\Delta t^4\right),
    \qquad 0\leq n\leq M.
\end{equation}
\end{theorem}

\begin{proof}
By the triangle inequality,
\[
    \|u(t_n)-u_N^n\|
    \leq
    \|u(t_n)-u_N(t_n)\|
    +
    \|u_N(t_n)-u_N^n\|.
\]
The first term is bounded by Theorem \ref{thm:spectral_accuracy}, and the
second term is bounded by Theorem \ref{thm:temporal_error_ifrk4}. This proves
\eqref{eq:fully_discrete_error}. The analytic estimate
\eqref{eq:fully_discrete_exp} follows in the same way from the exponential
spatial estimate in Theorem \ref{thm:spectral_accuracy}.
\end{proof}



\section{Numerical Experiments}
\label{sec6}

This section presents numerical validation of the FSG method for the fZK equation \eqref{eqn:fZK}. The computational methodology combines the spatial discretization from Section \ref{sec3} with the temporal scheme from Section \ref{sec5}, providing comprehensive verification of theoretical results and demonstrating the capabilities of scheme \eqref{fsg:fZK} across various fractional orders.

Although the analysis above is written for $\Omega=[-\pi,\pi]^2$, the same Fourier--Galerkin formulation applies to any rectangular periodic box after the usual affine rescaling of the Fourier frequencies. In the experiments below, large periodic boxes are used when localized solitary waves are tested. These computations should therefore be interpreted as large-period approximations of whole-space solitary waves, and the reported error contains a small periodization effect in addition to the spectral and temporal discretization errors. Unless stated otherwise, the nonlinear term is evaluated by the exact Galerkin convolution. 

For temporal discretization, the IFRK4 method \eqref{fdscheme} is implemented. The conservation laws proved in Theorem \ref{conslemma} hold exactly for the semi-discrete continuous-in-time Galerkin system. The fully discrete IFRK4 method is not an invariant-preserving time integrator, and therefore the invariants are monitored through the relative errors
\[
    \frac{|\mathcal I(u_N^n)-\mathcal I(u_N^0)|}{1+|\mathcal I(u_N^0)|},
    \qquad
    \frac{|\mathcal M(u_N^n)-\mathcal M(u_N^0)|}{1+|\mathcal M(u_N^0)|},
    \qquad
    \frac{|\mathcal H(u_N^n)-\mathcal H(u_N^0)|}{1+|\mathcal H(u_N^0)|}.
\]
The reported conservation plots therefore measure the observed fully discrete drift, not an exact fully discrete conservation law.

\subsection*{Example 1: Validation of Spectral Convergence via Solitary Wave Propagation}

This example serves to validate the theoretical framework developed in Sections \ref{sec3}-\ref{sec5} by examining the classical ZK equation \eqref{eqn:fZK} with $\alpha=2$, for which exact analytical solitary wave solutions are known. The solitary wave solution takes the form \cite{ZakharovKuznetsov1974,xu2005local}:
\begin{equation}\label{1solwavezk}
u(x,y,t) = 3c \sech^2\left(\frac{\sqrt{c}}{2}\left[(x - ct)\cos\theta + y\sin\theta\right]\right),
\end{equation}
representing a localized wave packet that propagates with constant velocity $c$ while maintaining its shape due to the precise balance between nonlinear steepening and dispersive spreading.
We consider a solitary wave with amplitude parameter $c=1$ propagating along the x-axis ($\theta=0$) over the time interval $t\in[0,10]$ on the spatial domain $\Omega = [-20\pi, 20\pi]^2$.

The spatial discretization employs the FSG method \eqref{fsg:fZK} with Fourier cutoffs ranging from $N=16$ to $N=256$ in each spatial direction. The time integration utilizes the fourth-order integrating factor Runge-Kutta (IFRK4) method \eqref{eq:rk4_update}, with the time step chosen as $\Delta t = 1/(N\|u_0\|_{L^{\infty}(\Omega)})$ to maintain numerical stability while ensuring that temporal errors remain subordinate to spatial discretization errors.
\begin{table}[h!]
\centering
\begin{tabular}{||c|c|c|c|c||}
\hline
$N$ & $L^2$ Error & Order & $L^\infty$ Error & Order \\
\hline\hline
16  & $2.922\times 10^{1}$ & --- & $1.384\times 10^{0}$ & --- \\
32  & $1.778\times 10^{1}$ & 0.72 & $1.109\times 10^{0}$ & 0.32 \\
64  & $6.248\times 10^{-1}$ & 4.83 & $5.398\times 10^{-2}$ & 4.36 \\
128 & $9.401\times 10^{-4}$ & 9.38 & $5.844\times 10^{-5}$ & 9.85 \\
256 & $3.798\times 10^{-8}$ & 14.60 & $6.017\times 10^{-9}$ & 13.25 \\
\hline
\end{tabular}
\caption{Spectral convergence analysis for the fZK equation \eqref{eqn:fZK} for $\alpha=2$ with exact solitary wave solution at $T=10$. The exponential convergence with decay rate $\sigma \approx 0.09$ demonstrates the characteristic spectral accuracy of the Fourier method, with errors decreasing by over nine orders of magnitude from $N=16$ to $N=256$.}
\label{tab:convergence}
\end{table}

The numerical results presented in Figure \ref{fig:solitary_wave} and Table \ref{tab:convergence} validate the spatial accuracy of the method in the classical case $\alpha=2$. Since the solitary wave is a whole-space solution, the computation on the large periodic box $\Omega=[-20\pi,20\pi]^2$ is a large-period approximation. For sufficiently large \(N\), the numerical solution accurately captures the wave profile at \(T=10\), and the errors exhibit the rapid decay expected from Fourier approximation of smooth localized data. The cross-sectional view in Figure \ref{fig:solitary_wave}(a) demonstrates the exact preservation of the solitary wave structure during propagation, with the numerical solution at $T=10$ showing that the approximate solution matches the exact solution. The two-dimensional visualizations in Figures \ref{fig:solitary_wave}(b) and \ref{fig:solitary_wave}(c) further confirm the method's accuracy in capturing the spatial structure of the solitary wave. Figure \ref{fig:solitary_wave}(d) verifies the spectral convergence of the scheme. The error evolution in Table \ref{tab:convergence} exhibits the  exponential convergence of spectral methods, with the $L^2$ error decreasing from $\mathcal{O}(10^{1})$ at $N=16$ to $\mathcal{O}(10^{-8})$ at $N=256$. 

The figure \ref{fig:cons} demonstrates that for the classical case $(\alpha=2)$ where exact solitary wave solutions exist, the FSG method \eqref{fsg:fZK} preserves the fundamental invariants \eqref{convquant} with exceptional precision over extended time evolution $(T=1000)$. The relative conservation errors at final time remain within $\mathcal{O}(10^{-14})$ for mass, $\mathcal{O}(10^{-08})$ for momentum, and $\mathcal{O}(10^{-08})$ for Hamiltonian energy, validating both the theoretical conservation properties and numerical implementation accuracy.

\begin{figure}[h!]
\centering
\includegraphics[width=1\textwidth]{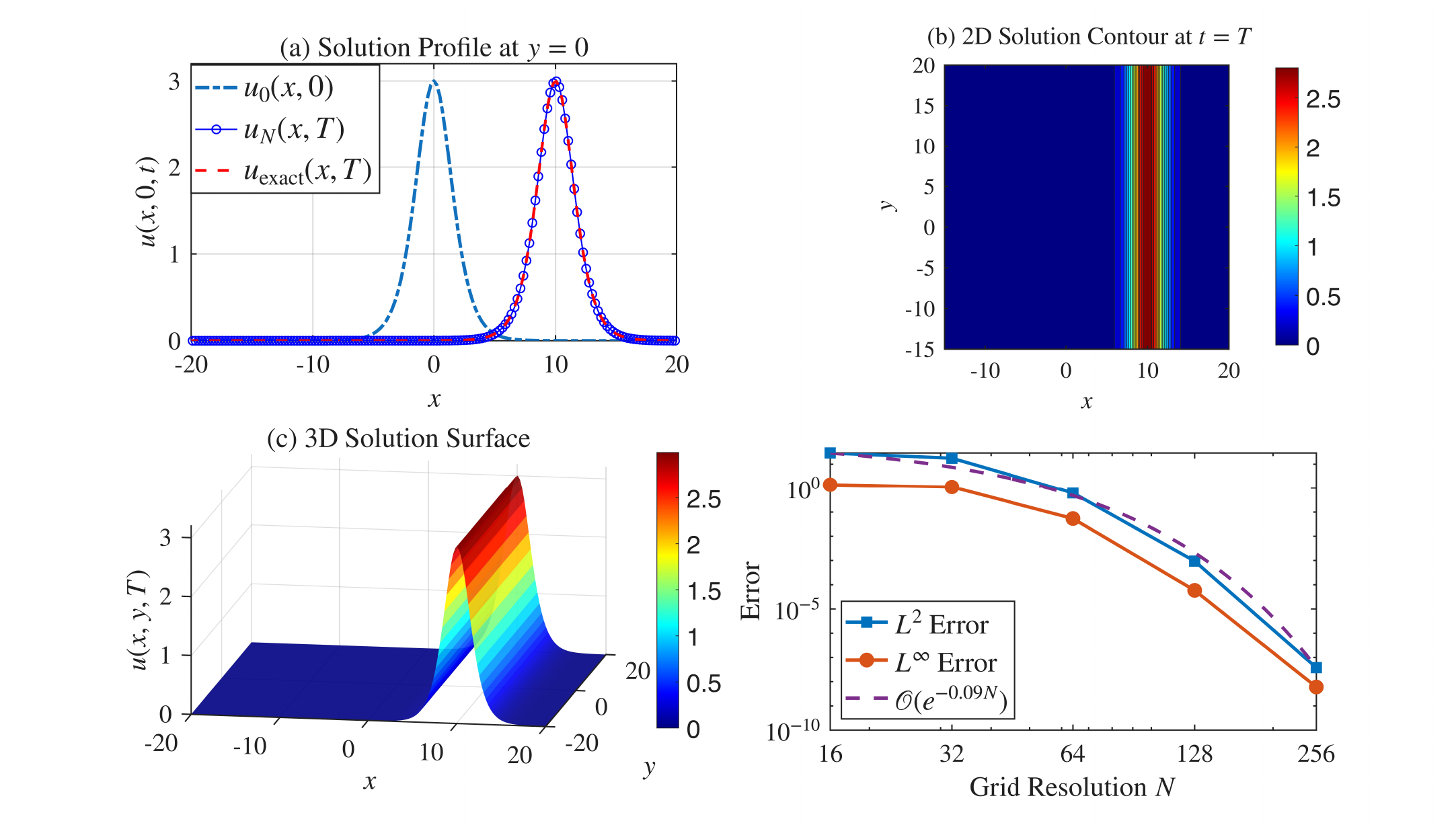}
\caption{Comprehensive validation of the FSG method \eqref{fsg:fZK} for the fZK equation ($\alpha=2$) \eqref{eqn:fZK}. Panel (a) shows the cross-sectional wave profile at $y=0$. Panel (b) displays the 2D contour plot. Panel (c) provides a 3D surface visualization of the solitary wave profile. Panel (d) presents the spectral convergence analysis. Parameters: $c=1$, $\theta=0$, $T=10$, domain $\Omega = [-20\pi,20\pi]^2$ and time step $\Delta t = 1/(N\|u\|_{L^{\infty}(\Omega)})$.}
\label{fig:solitary_wave}
\end{figure}

\begin{figure}[h!]
\centering
\includegraphics[width=1\textwidth]{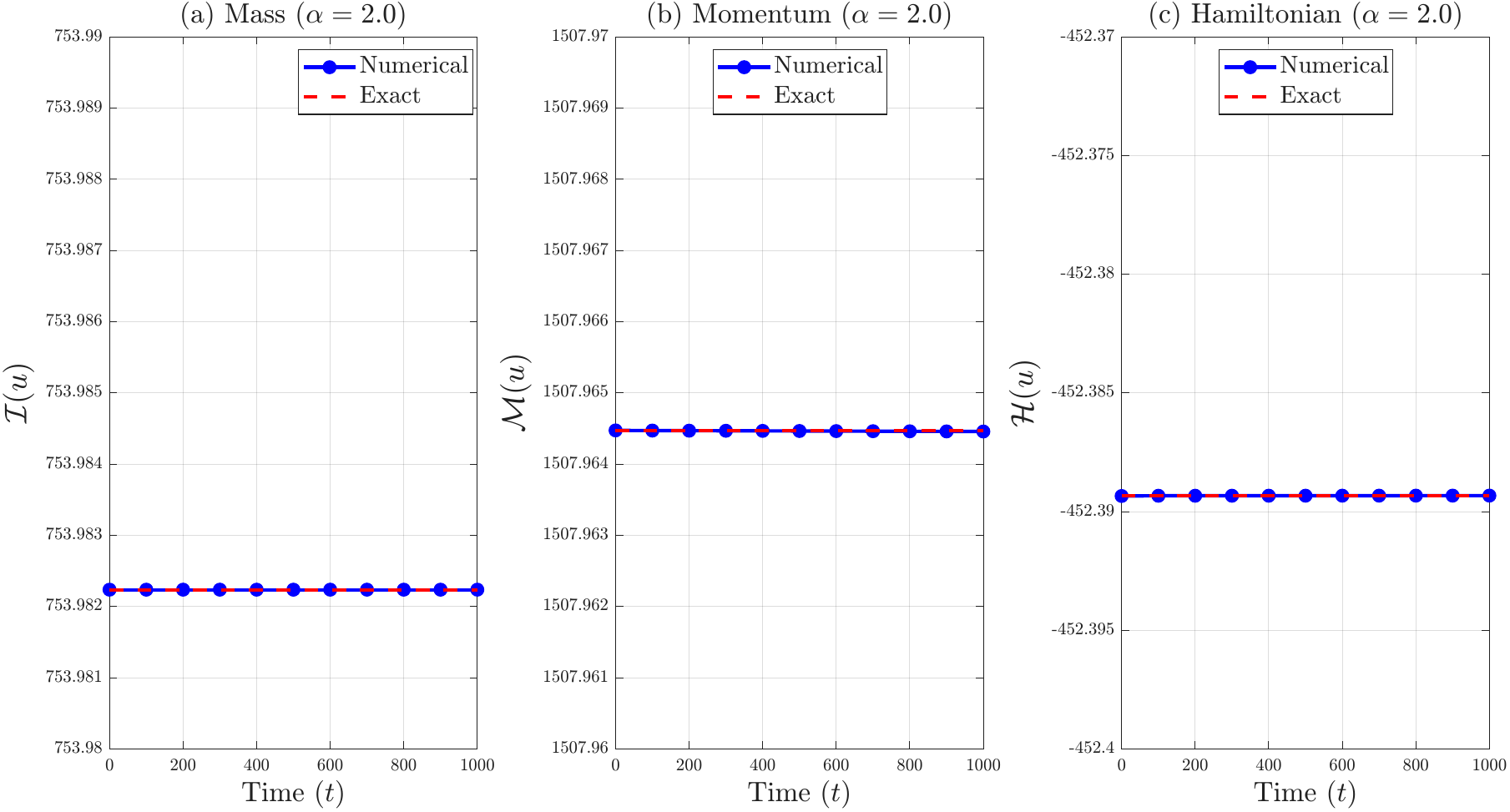}
\caption{Observed long-time relative conservation errors for the fully discrete IFRK4--FSG approximation \eqref{fdscheme} of the fZK equation \eqref{eqn:fZK} with $\alpha=2$. The semi-discrete Galerkin system conserves $\mathcal I$, $\mathcal M$, and $\mathcal H$ exactly, but the IFRK4 time discretization is not exactly invariant-preserving. The plotted quantities therefore show the fully discrete drift of the invariants over $t\in[0,1000]$. The numerical solution (blue circles) maintains near-machine-precision agreement with the exact conserved quantities (red dashed lines), with relative errors of $\mathcal{O}(10^{-14})$, $\mathcal{O}(10^{-08})$, and $\mathcal{O}(10^{-08})$ respectively, confirming the small fully discrete drift of the invariants for the chosen resolution and time step for the classical ZK equation \eqref{zkeqn}.}
\label{fig:cons}
\end{figure}

\subsection*{Example 2. (Soliton Interactions)}
The interaction of solitary waves provides a rigorous test of numerical methods for nonlinear wave equations, particularly for assessing a scheme's ability to handle nonlinear interactions while preserving structural properties. Following the approach in \cite{xu2005local,SunLiuWei2016}, we consider the two-soliton initial condition:
\begin{equation}
u_0(x,y) = \sum_{j=1}^2 3c_j \sech^2\left(\frac{\sqrt{c_j}}{2}\left[(x-x_j)\cos\theta_j + (y-y_j)\sin\theta_j\right]\right)
\end{equation}
with parameters $c_1=0.5$, $c_2=0.2$, $\theta_1=0$, $\theta_2=0$, $x_1=-15$, $x_2=0$, and $y_1=y_2=0$. The computational domain is $\Omega = [-20\pi, 20\pi]^2$ with $N=512$ in each direction and a time step $\Delta t = 1/(N\|u_0\|_{L^{\infty}(\Omega)})$ over the time interval $t \in [0,60]$.

\begin{figure}[h!]
\centering
\includegraphics[width=1\textwidth]{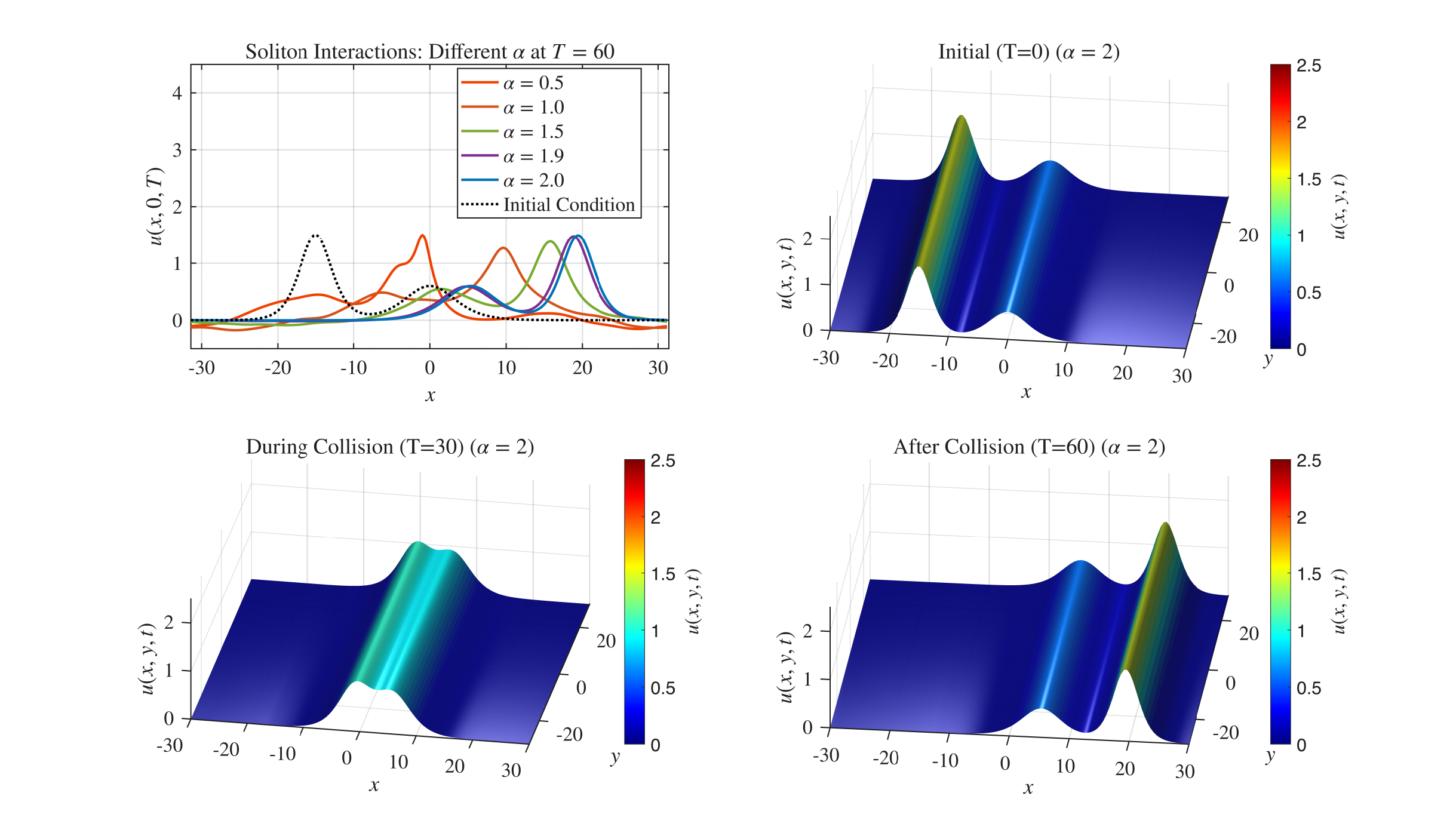}
\caption{Soliton interaction dynamics for the fZK equation ($\alpha=2$). Panel (a) shows the wave profile at $y=0$ for different fractional orders $\alpha\in\{0.5, 1, 1.5, 1.9, 2.0\}$ at final time $T=60$. Panel (b) displays the 3D surface visualization of the initial two-soliton configuration ($T=0$). Panel (c) provides a 3D surface visualization during collision ($T=30$). Panel (d) presents the 3D surface visualization after collision ($T=60$). Parameters: $c_1=0.5$, $c_2=0.2$, $\theta_1=0$, $\theta_2=0$, domain $\Omega = [-20\pi,20\pi]^2$, spatial resolution $N=512$, and time step $\Delta t = 1/(N\|u_0\|_{L^{\infty}(\Omega)})$.}
\label{fig:soliton_interaction}
\end{figure}

Figure \ref{fig:soliton_interaction} presents a comprehensive analysis of soliton interactions across different fractional orders. Panel (a) reveals the significant impact of fractional order $\alpha$ on soliton dynamics: as $\alpha \in \{0.5, 1, 1.5, 1.9 ,2 \}$, the enhanced dispersion leads to faster propagation speeds and more pronounced wave spreading. The larger-amplitude soliton ($3c_1=1.5$) consistently overtakes the smaller one ($3c_2=0.6$) across all fractional orders, but the interaction characteristics vary substantially with $\alpha$. For lower fractional orders ($\alpha=1.2$), the reduced dispersion results in more localized interaction and slower propagation, while higher orders ($\alpha=2.0$) exhibit stronger dispersive effects that accelerate the solitons and modify the interaction profile. Panels (b)-(d) detail the complete interaction process for the classical ZK equation ($\alpha=2.0$). Panel (b) shows the initial configuration at $T=0$ with two well-separated solitons propagating along the directions $\theta_1=0$ and $\theta_2=0$. Panel (c) captures the nonlinear interaction phase at $T=30$, where the larger soliton overtakes the smaller one, creating a complex merged wave structure. Panel (d) demonstrates the post-collision state at $T=60$, showing both solitons emerging with their original amplitudes and shapes preserved. The post-interaction profiles suggest that the method resolves the localized wave interaction without significant numerical distortion over the simulated
time interval.

The numerical simulation confirms several key physical and mathematical properties and is consistent with the theoretical prediction that soliton velocity scales linearly with amplitude in the ZK equation \cite{ZakharovKuznetsov1974}. This example validates the method's capability for simulating complex nonlinear wave phenomena and provides confidence in its application to more challenging problems in plasma physics and nonlinear wave theory.

The numerical experiments comprehensively validate the theoretical analysis and demonstrate the effectiveness of the IFRK4-FSG method \eqref{fdscheme} for the fZK equation \eqref{eqn:fZK}. The scheme accurately captures nonlinear wave phenomena, preserves important structural properties, and provides efficient computation across various fractional orders.

\section{Conclusion}
\label{sec7}
This work introduced a Fourier spectral Galerkin method for the periodic fZK
equation \eqref{eqn:fZK} in the full fractional range \(0<\alpha\leq2\).
This includes the classical ZK equation \(\alpha=2\), the
higher-dimensional Benjamin--Ono--ZK endpoint \(\alpha=1\), and weaker
fractional-dispersion regimes \(0<\alpha<1\). The semi-discrete scheme
preserves discrete analogues of mass, momentum, and Hamiltonian energy. Using
periodic Kato--Ponce and commutator estimates, we established local-in-time
Sobolev bounds and proved strong convergence of the Galerkin approximations
to the unique local strong solution in the high-regularity class
\(H^s_{\mathrm{per}}(\Omega)\), \(s>2+\alpha\). We also introduced a modified
projection adapted to the fractional transport-dispersive operator and proved
optimal spatial convergence for \(r>2+\alpha\). A refined adjoint multiplier
estimate is the key technical ingredient that allows the optimal projection
argument to remain valid also for \(0<\alpha<1\). Finally, an
integrating-factor RK4 time discretization was analyzed, yielding fourth-order
temporal convergence under a high-regularity nonlinear stability assumption.

\bibliographystyle{abbrv}
\bibliography{main}

\end{document}